\documentclass[10pt]{amsart}
\usepackage{amsmath,amssymb,mathtools,mathabx,rotating}
\usepackage{graphicx}
\usepackage{color}
\usepackage{enumerate}
\usepackage{epstopdf}
\newcommand*{\id}{{\mathrm{id}}}

\newcommand*{\un}{{\mathbf 1}}

\newcommand*{\End}{{\mathrm{End}}}

\def\shuff#1#2{\mathbin{
      \hbox{\vbox{\hbox{\vrule \hskip#2 \vrule height#1 width 0pt}\hrule}\vbox{\hbox{\vrule \hskip#2 \vrule height#1 width 0pt\vrule }\hrule}}}}
\def\shuffl{{\mathchoice{\shuff{5pt}{3.5pt}}{\shuff{5pt}{3.5pt}}{\shuff{3pt}{2.6pt}}{\shuff{3pt}{2.6pt}}}}
\def\shuffle{{\, \shuffl \,}}

\def\boxsucc {\Box \hspace{-0.38cm} \succ}
\def\bboxsucc {\Box \hspace{-0.23cm} \succ}
\def\boxprec{\Box \hspace{-0.38cm} \prec}
\def\bboxprec {\Box \hspace{-0.23cm} \prec}

\newtheorem{thm}{Theorem}

\newtheorem{cor}[thm]{Corollary}

\newtheorem{lem}[thm]{Lemma}
\newtheorem{prop}[thm]{Proposition}
\newtheorem{defn}[thm]{Definition}
\newtheorem{rmk}[thm]{Remark}

\begin{document}
 
\title[Shuffle groups and free probability]{Shuffle group laws. \\ Applications in free probability}

\vspace{1cm}

\author{Kurusch Ebrahimi-Fard}
\address{Department of Mathematical Sciences, 
		Norwegian University of Science and Technology (NTNU), 
		7491 Trondheim, Norway.}
\email{kurusch.ebrahimi-fard@ntnu.no}         
\urladdr{https://folk.ntnu.no/kurusche/}

\author{Fr\'ed\'eric Patras}
\address{Univ.~C\^ote d'Azur, CNRS,
         		UMR 7351, 
         		Parc Valrose,
         		06108 Nice Cedex 02, France.}
\email{patras@math.unice.fr}
\urladdr{www-math.unice.fr/$\sim$patras}


\begin{abstract}
Commutative shuffle products are known to be intimately related to universal formulas for products, exponentials and logarithms in group theory as well as in the theory of free Lie algebras, such as, for instance, the Baker--Campbell--Hausdorff formula or the analytic expression of a Lie group law in exponential coordinates in the neighbourhood of the identity. Non-commutative shuffle products happen to have similar properties with respect to pre-Lie algebras. However, the situation is more complex since in the non-commutative framework three exponential-type maps and corresponding logarithms are naturally defined. This results in several new formal group  laws together with new operations -- for example a new notion of adjoint action particularly well fitted to the new theory. These developments are largely motivated by various constructions in non-commutative probability theory. The second part of the article is devoted to exploring and deepening this perspective. We illustrate our approach by revisiting universal products from a group-theoretical viewpoint, including additive convolution in monotone, free and boolean probability, as well as the Bercovici--Pata bijection and  subordination products.
\end{abstract}

\maketitle

\tableofcontents

\begin{quote}
\end{quote}


\section{Introduction}
\label{sec:intro}

The formal, i.e., purely algebraic component of the relations between a Lie group and its Lie algebra is encapsulated in the Baker--Campbell--Hausdorff formula:
$$
	x \ast y = \log (\exp x \cdot \exp y) 
	            =:{\mathrm{BCH}}(x,y)
$$
as well as in the formal expression of the Lie group law in exponential coordinates in the neighbourhood of the identity element. In practice, as emphasised, for example, in Reutenauer's classical monograph on free Lie algebras \cite{reutenauer}, the best framework to understand these phenomena is provided by the theory of Hopf algebras. Indeed, in the complete connected cocommutative case, the exponential and logarithm maps relate the Lie algebra of primitive elements bijectively to the group of group-like elements \cite{quillen}. Here, classical commutative shuffle algebra is an essential ingredient for at least two reasons. First, the dual of the enveloping algebra of a free Lie algebra is a free commutative shuffle algebra \cite{fpshuffle,schutz}. Second, the combinatorics of classical commutative shuffle products, as they appear, for example, in cards shuffling in probability, is encoded through the descent algebra. The latter forms a subalgebra of the direct sum of the symmetric group algebras (equipped with a suitable convolution product), that identifies with an algebra of natural transformations of the forgetful functor from graded connected cocommutative Hopf algebras to graded vector spaces, and whose properties can be used to understand relations between group-like and primitive elements~\cite{patras}.

The present paper originates from a seemingly different subject, i.e., the moment-cumulant relations in Voiculescu's theory of free probability \cite{mingospeicher_17,SpeicherNica,voiculescu_92,voiculescu_95}. Free cumulants encode the notion of free independence in free probability. Analogous statements apply to monotone and boolean cumulants, and the respective notions of independence between random variables \cite{hasebesaigo_11,speicher_97a}. In the 1990's, Speicher uncovered a combinatorial approach to free cumulants and to the notion of freeness in Voiculescu's theory, which is based on non-crossing set partitions and M\"obius calculus \cite{biane_02,speicher_94,speicher_97c,speicher_98}. This approach extends to the boolean and monotone cumulants, and the relations between different types of cumulants can be encoded in terms of various  non-crossing set partitions. See Arizmendi et al.~\cite{lehner-etal_15} for details. 

In \cite{ebrahimipatras_15,ebrahimipatras_16,ebrahimipatras_17}, the authors followed a radically different path by  exploring monotone, free, and boolean cumulants from the point of view of non-commutative shuffle algebras. The starting point is a graded connected non-commutative non-cocommutative Hopf algebra defined on the double tensor algebra $H$ over a non-commutative probability space $(\mathcal A,\varphi)$. The coproduct on $H$ splits into the sum of two so-called ``half-coproducts''. On the dual space $H^\ast$ this results in the splitting of the convolution product into two ``half-products'', which defines a non-commutative shuffle algebra on $H^\ast$. Notice that, as usual for duals of Hopf algebras, $H^\ast$ contains the group of characters together with the Lie algebra of infinitesimal characters on $H$. However, due to the splitting of the convolution product, both the group and the Lie algebra carry extra structure beyond the classical setting. This provides the key to our forthcoming developments. For details the reader is referred to the first part of the present article, where definitions are given.

The central observation in reference \cite{ebrahimipatras_17} is that monotone, free, and boolean moment-cumulant relations can be described in terms of three different exponential-type maps associated respectively to the convolution product and the two half-products defined on the dual $H^\ast$. Here, the notion of exponential-type map generalises  the one of time-ordered exponential in physics \cite{raph,ebrahimipatras_14}. Logarithm-type maps corresponding to these exponential-type maps can be defined, and shuffle algebra identities permit to express monotone, free, and boolean cumulants in terms of each other using the Magnus expansion familiar in the context of pre-Lie algebras. 

The present paper addresses two questions. The first builds on the works of Chapoton \cite{chapoton}, Foissy \cite{foissy}, Ronco \cite{ronco1,ronco}, and others \cite{ebrahimimanchon_09}, that studied the algebraic structure and properties of shuffle algebras and bialgebras. Notice that in these references, shuffle (bi)algebras are called dendriform (bi)algebras, and commutative shuffle algebras are called Zinbiel algebras. For theoretical and historical reasons we prefer the name shuffle (bi)algebra, as it is grounded in classical works in algebraic topology, combinatorics, and control theory, and echoes the seminal Eilenberg--MacLane--Sch\"utzenberger idea of splitting commutative and non-commutative shuffle products into two ``half-products''. Here, we explore how the classical correspondence between Lie groups and Lie algebras, and related properties and identities, translate in the non-(co)commutative setting, that is, when one considers instead of the usual exponential and logarithm maps their three shuffle bialgebra counterparts. A rather rich structure arises from this approach. For example, adjoint actions and several universal group laws can be defined at the group and Lie algebra level. The latter generalises, for example, the classical group law induced on a free Lie algebra from the associated exponential group and the Baker--Campbell--Hausdorff formula.

The second question is based on our recent work \cite{ebrahimipatras_17}, and its answer illustrates the meaningfulness of the shuffle algebra approach in free probability. Namely, we explore the notion of additive convolution \cite{lenczewski_07,nica_09} from the shuffle algebra viewpoint. Our approach is group-theoretical and relates to the notion of universal products \cite{bengohrschur_02,muraki_03,speicher_97b}. Instead of viewing additive free, monotone or boolean convolution as operations on distributions of non-commutative random variables, we lift these operations to appropriate groups of characters (which are affine pro-algebraic groups). In the particular case of the Hopf algebra $H$ mentioned above, this yields a way of defining the ``additive convolution'' of two linear forms $\varphi$ and $\varphi'$ on a given associative algebra $\mathcal A$. This construction at the group level specialises then to formulas for the joint distributions of non-commutative random variables. As an application, we revisit the Belinschi--Nica semigroup and the corresponding boolean Bercovici--Pata bijection as well as additive convolutions of distributions. 

A remark is in order regarding the appearance of shuffle as well as Hopf algebras in the context of non-commutative probability theory. Several works have applied Hopf algebra techniques to free probability. Besides the classical work of Mastnak and Nica \cite{mastnaknica_10}, let us mention the connections of free probability with Witt vectors, lambda rings and related topics developed by Friedrich and McKay  in \cite{friedrichmckay_12, friedrichmckay_13}. The same authors also  proposed recently in \cite{friedrichmckay_15} another approach to non-commutative probability and convolutions that relies on the notion of cogroups in the category of associative algebras (recall that commutative Hopf algebras are cogroups in the category of commutative algebras). Manzel and Sch\"urmann in \cite{manzelschuermann_16} used combinatorial Hopf algebra to explore the notion of non-commutative stochastic independence from a cumulant perspective. Shuffle (or dendriform) algebras appeared already in \cite{belinschietal_09} in relation to the problem of infinite divisibility with respect to additive convolution in free probability. Although potential connections deserve to be explored, the present article (as well as the references \cite{ebrahimipatras_15,ebrahimipatras_16,ebrahimipatras_17}) follows different directions. 

\smallskip 

The paper is organized as follows. In the next sections we recall the notions of (un)shuffle algebras and bialgebras, and introduce the fundamental example in view of applications to monotone, free, and boolean probabilities. In Section \ref{sect:group} we define algebraic shuffle groups and Lie algebras, and recall from \cite{ebrahimipatras_17} the construction of exponential and logarithm-type maps in this context together with some identities they satisfy. Section \ref{sect:automorphisms} studies the various natural set automorphisms of shuffle Lie algebras induced by the three exponential and logarithm-type maps. It emphasises connections with the theory of pre-Lie algebras -- especially the pre-Lie Magnus expansion and its inverse. In Section \ref{sect:shuffleadjoint} we introduce three adjoint actions, which are defined respectively in terms of the shuffle and the two half-shuffle products. Section \ref{sec:BCHcumulants} explores universal group laws that generalise the Baker--Campbell--Hausdorff group law on a free Lie algebra. In Section \ref{sec:convol} we define left and right additive convolutions as commutative group laws on groups of characters. Section \ref{sect:universal} explores the products from the foregoing section from the point of view of non-commutative probability. It is shown that the three universal products in non-commutative probability, which are related to free, monotone and boolean additive convolution, can be obtained directly from those developed in the previous sections. In Section \ref{sect:BPbijection} it is shown how the point of view of half-shuffle logarithms and exponentials relates to the Bercovici--Pata bijection. Section \ref{sect:additiveconvol} provides an algebraic point of view on free additive convolution of distributions of non-commutative random variables. 

\smallskip

For background on coalgebras, bialgebras and Hopf algebras, we refer to \cite{cartier_07, figueroagraciabondia_05, manchon_08}. All algebraic structures are defined over the ground field $k$ of characteristic zero. They are graded if their structure maps are defined in the tensor category of graded vector spaces, that is, if objects are graded vector spaces and the structure maps respect the grading. Connectedness in the graded context refers to the degree zero component being isomorphic to the ground field. We also assume any $k$-(co)algebra to be (co)associative, if not stated otherwise. 

\vspace{0.4cm}

\noindent {\bf{Acknowledgements}}: We thank Roland Speicher for stimulating discussions. The second author acknowledges support from the grant ANR-12-BS01-0017, ``Combinatoire Alg\'ebrique, R\'esurgence, Moules et Applications". Support by the CNRS GDR ``Renormalisation" and the PICS program CNRS/CSIC, JAD-ICMAT ``Alg\`ebres de Hopf combinatoires et probabilit\'es non commutatives" is also acknowledged.


\section{Unshuffle bialgebras}
\label{sect:shuffle}

Recall that an (affine) algebraic group $G$ over a ground field $k$ is a functor from commutative algebras ${A}$ over $k$ to groups such that, up to equivalence, $G({A})$ is the group of characters of a commutative Hopf algebra $H$, i.e., the group of algebra maps from $H$ to ${A}$, where the group law is induced by the coproduct on $H$. See, for instance, Cartier \cite{cartier_07}. A compact Lie group is always algebraic, that is, isomorphic to the group of real points of an algebraic group, whereas an arbitrary Lie group is pro-algebraic. We shall follow this point of view, as it is convenient with respect to the definition of groups associated to the notion of shuffle algebra as well as their application in free probability. 

\smallskip

Whereas the abstract notion of shuffle product and its decomposition into half-shuffles, which will be recalled further below, goes back to the works of Eilenberg, MacLane and Sch\"utzenberger in the 1950's, the dual notion seems to have been considered only relatively recently from an abstract, i.e., axiomatic point of view. For details on shuffle and unshuffle bialgebras we refer the reader to the works of Chapoton \cite{chapoton}, Foissy \cite{foissy} and Ronco \cite{ronco1,ronco}.

Recall that a $k$-coalgebra $\overline C$ is coaugmented if it is equipped with a coalgebra map from the ground field $k$ to $\overline C$. A counital coaugmented coalgebra splits as $\overline C = k \oplus C$, where $C$ is the kernel of the counit.

\begin{defn} \label{def:unshufflecoalg}
A counital unshuffle (aka codendriform) coalgebra is a coaugmented coassociative coalgebra $\overline C = k \oplus C$ with  coproduct
\begin{equation}
\label{codend}
	\Delta(c) := \bar\Delta(c) + c \otimes \un + \un \otimes c,
\end{equation}
such that the reduced coproduct on $C$ splits, $\bar\Delta = \Delta_{\prec} + \Delta_{\succ}$, with 
\begin{eqnarray}
	(\Delta_{\prec} \otimes \id) \circ \Delta_{\prec}   &=& (\id \otimes \bar\Delta)\circ \Delta_{\prec}        	\label{C1}\\
  	(\Delta_{\succ} \otimes \id) \circ \Delta_{\prec}   &=& (\id \otimes \Delta_{\prec})\circ \Delta_{\succ} 	\label{C2}\\
   	(\bar\Delta \otimes  \id) \circ \Delta_{\succ}         &=& (\id \otimes \Delta_{\succ})\circ \Delta_{\succ}   \label{C3}.
\end{eqnarray}
\end{defn}

\noindent The maps $\Delta$, $\Delta_{\prec}$ and $\Delta_{\succ}$ are called respectively unshuffling coproduct and {\it{left}} and {\it{right half-unshuffles}}. The definition of a non-unital unshuffle coalgebra is obtained by removing the unit, that is, $\bar\Delta$ is  acting on $C$, and has a splitting into two half-coproducts, $\Delta_{\prec}$ and $\Delta_{\succ}$, which obey relations (\ref{C1}), (\ref{C2}) and (\ref{C3}).

\begin{defn}
An unshuffle (aka codendriform) bialgebra is a unital and counital bialgebra $\overline B= k \oplus B$ with product $\cdot_B$ and coproduct $\Delta$; a counital unshuffle coalgebra with $\bar\Delta = \Delta_{\prec} + \Delta_{\succ}$, such that, moreover, the following compatibility relations hold 
\begin{eqnarray}
	\Delta^+_{\prec}(a \cdot_B b)  &=& \Delta^+_{\prec}(a)  \cdot_B \Delta(b)      	\label{D1}\\
  	\Delta^+_{\succ}(a \cdot_B b)  &=& \Delta^+_{\succ}(a)  \cdot_B \Delta(b),     \label{D2}
\end{eqnarray}
where
\begin{eqnarray}
	\Delta^+_{\prec}(a)  &:=& \Delta_{\prec}(a) + a \otimes \un     	\label{D3}\\
  	\Delta^+_{\succ}(a)  &:=& \Delta_{\succ}(a) + \un \otimes a.     	\label{D4}
\end{eqnarray}
\end{defn}

The most classical example of an unshuffle bialgebra is the unital tensor algebra $\overline T(X)$ over an alphabet $X = \{x_1,x_2,\ldots \}$. Notice, for later use, that we write $T(X)$ for the non-unital tensor algebra. Both are the linear span of words $x_{i_1} \cdots x_{i_n}$ over $X$, with the empty word included in the unital case, but not in $T(X)$. The concatenation product of two words, $x_{i_1}\cdots x_{i_n} \cdot x_{j_1}\cdots x_{j_m} := x_{i_1}\cdots x_{i_n}x_{j_1}\cdots x_{j_m}$ turns ($\overline T(X)$) $T(X)$ into a (unital) non-commutative $k$-algebra. The unital tensor algebra $\overline T(X)$ is endowed with the unshuffling coproduct
$$
	\Delta(x_{i_1} \cdots x_{i_n}) := \sum_{I \coprod J = [n]} x_I \otimes x_J,
$$
where $x_S$ stands for the word $x_{i_{s_1}}\cdots x_{i_{s_k}}$ associated to the naturally ordered subset $S=\{s_1,\ldots ,s_k\} \subseteq [n]$. 
Setting 
$$
	\Delta_\prec^+(x_{i_1} \cdots x_{i_n})
		:=\sum_{I\coprod J=[n] \atop 1\in I}x_I\otimes x_J,
$$ 
and $\Delta_\succ^+:=\Delta - \Delta_\prec^+$, defines an unshuffle bialgebra structure on $\overline T(X)$, whose fine structure is studied in \cite{fpshuffle}. 

Recall that the notions of bialgebra and Hopf algebra identify when suitable connectedness hypothesis hold, for example, when the bialgebra can be equipped with a grading such that the degree zero component is the ground field \cite{cartier_07}. This hypothesis is satisfied for $\overline{T}(X)$ and will also be fulfilled in all the other examples we will consider, so that there is no difference in this article between bialgebras and Hopf algebras.

It is well-known that classical graded connected cocommutative Hopf algebras over a field of characteristic zero are enveloping algebras of their primitive part (Cartier--Milnor--Moore theorem) \cite{cartier_07,patras}. This structure theorem has a dual, slightly weaker version, known as Leray's theorem. It states that a graded connected commutative Hopf algebra $H$ over a field of characteristic zero is a polynomial algebra over the vector space of indecomposable elements. Recall that the latter is the quotient $H^+/(H^+)^2$ -- it is not a subspace of $H^+$, the augmentation ideal of $H$, but the techniques developed in \cite{patras} allow to construct a canonical lift from $H^+/(H^+)^2$ to $H^+$, and therefore permit to consider slightly abusively the space of indecomposable elements as a subspace of $H$.

Similar statements hold \it mutatis mutandis \rm for graded connected unshuffle bialgebras and their dual shuffle bialgebras (the axioms for shuffle algebras are recalled below, in Section \ref{sect:group}). These results are due to Chapoton \cite[Prop.~14, Prop.~15]{chapoton}. A more combinatorially flavoured approach was developed in parallel by Ronco \cite{ronco1}. Their work triggered much of the interest that has developed around the notion of shuffle algebra in the theory of operads and in algebraic combinatorics in recent years. We mention the analog of Leray's theorem in our context:

\begin{thm}\label{chapo}
Let $\overline B$ be a graded connected unshuffle bialgebra over $k$, then $B$ is, as an associative algebra, naturally isomorphic to $T(I)$, the free associative algebra over a subspace $I$ of $B$, which is canonically isomorphic to the vector space of indecomposable elements $B/B^2$. 
\end{thm}

Notice that Chapoton's and Ronco's results are stated in the dual framework, the Theorem follows immediately by duality. The fact that there is a canonical lift from $B/B^2$ to $I$ follows, for example, from the naturallity of the isomorphism in \cite[Prop.~15]{chapoton}. 

For completeness, we recall from \cite{gv} that a brace algebra is a vector space $V$ equipped with $n+1$-ary multilinear operations, $n\geq 0$, denoted $\{- ;\ - ,\dots,\ - \}$ such that $\{x;\emptyset\}:=x$ and
$$
	\{\{x;y_1,\dots,y_n\};z_1,\dots,z_m\}=\sum\{x; \ldots,\{y_1;\dots \},\ldots,\{y_n;\dots\},\dots\},
$$
where the empty spaces in the sum on the right-hand side of the equality are filled in all possible ways with the $z_i$ (in the given order). There is a forgetful functor from shuffle to brace algebras and Chapoton's original Theorem states that (under connectedness assumptions) a shuffle bialgebra is always the enveloping algebra of a brace algebra -- where by enveloping algebra functor we mean the left adjoint to this  forgetful functor. We refrain from giving more details on these notions and results as they are not of direct use with respect to our further developments. One should, however, keep in mind that they are lurking in the background of the theory we develop.


\section{The key example}
\label{sec:key}

Let us return to examples of unshuffle bialgebras. The key example for our purpose is given in terms of the double tensor algebra -- or double bar construction -- over a (non-commutative probability space) $k$-algebra $\mathcal A$. Define $T(\mathcal A) := \oplus_{n > 0} \mathcal A^{\otimes n}$ to be the non-unital tensor algebra over $\mathcal A$. The unital tensor algebra is denoted $\overline T(\mathcal A) := \oplus_{n \ge 0} \mathcal A^{\otimes n}$. Elements in $T(\mathcal A)$ are written as words $a_1 \cdots a_n \in T(\mathcal A)$. Notice that  $a \cdots a \in \mathcal A^{\otimes n}$ stands for $a^{\otimes n}$. To avoid confusions, the product of the $a_i$s in $\mathcal A$ is written $a_1 \cdot_\mathcal A a_2\cdot_\mathcal A \dots \cdot_\mathcal A a_n$. Concatenation of words makes $T(\mathcal A)$ an algebra, which is naturally graded by the length of a word, i.e., its number of letters. We set $T(T(\mathcal A)) := \oplus_{n > 0} T(\mathcal A)^{\otimes n}$, and use the bar notation to denote elements in $T(T(\mathcal A))$: $w_1 | \cdots | w_n \in T(T(\mathcal A))$, $w_i \in T(\mathcal A)$, $i = 1,\ldots,n$. Given $a = w_1 | \cdots | w_n$ and $b=  w_1' | \cdots | w_m'$, their product in $T(T(\mathcal A))$ is written $a|b$. This algebra is multigraded, $T(T(\mathcal A))_{n_1,\ldots ,n_k} := T_{n_1}(\mathcal A)\otimes \cdots \otimes T_{n_k}(\mathcal A)$, as well as graded. The degree $n$ part is  $T(T(\mathcal A))_n := \bigoplus_{n_1+ \cdots +n_k=n}T(T(\mathcal A))_{n_1,\ldots ,n_k}$. Similar observations hold for the unital case $\overline T(T(\mathcal A))=\oplus_{n \ge 0} T(\mathcal A)^{\otimes n}$, and we will identify without further comments a bar symbol such as $w_1| 1 |w_2$ with $w_1|w_2$ (formally, using the canonical map from $\overline T(\overline T(\mathcal A))$ to $\overline T(T(\mathcal A))$).

\medskip

Given two (canonically ordered) subsets $S \subseteq U$ of the set of integers $\bf N^\ast$, we call connected component of $S$ relative to $U$ a maximal sequence $s_1, \ldots , s_n$ in $S$, such that there are no $ 1 \leq i < n$ and $t \in U$, such that $s_i < t <s_{i+1}$. In particular, a connected component of $S$ in $\bf N^\ast$ is simply a maximal sequence of successive elements $s,s+1,\ldots ,s+n$ in $S$. Consider a word $a_1\cdots a_n \in T(\mathcal A)$. For $S:=\{s_1,\ldots, s_p\} \subseteq [n]$, we set $a_S:= a_{s_1} \cdots a_{s_p}$ ($a_\emptyset:=\mathbf{1}$). Denoting $J_1,\ldots,J_k $ the connected components of $[n] - S$, we set $a_{J^S_{[n]}}:= a_{J_1} | \cdots | a_{J_k}$. More generally, for $S \subseteq U \subseteq [n]$, we set  $a_{J^S_U}:= a_{J_1} | \cdots | a_{J_k}$, where the $a_{J_j}$ are now the connected components of $U-S$ in $U$.

\begin{defn} \label{def:coproduct}
The map $\Delta : T(\mathcal A) \to \overline T(\mathcal A) \otimes  \overline T(T(\mathcal A))$ is defined by
\begin{equation}
\label{HopfAlg}
	\Delta(a_1\cdots a_n) := \sum_{S \subseteq [n]} a_S \otimes  a_{J_1} | \cdots | a_{J_k}
					   =\sum_{S \subseteq [n]} a_S \otimes a_{J^S_{[n]}},
\end{equation} 
with $\Delta(\un):= \un \otimes \un$. This map is then extended multiplicatively to a coproduct on $\overline T(T(\mathcal A))$
$$
	\Delta(w_1 | \cdots | w_m) := \Delta(w_1) \cdots \Delta(w_m).
$$
\end{defn}
 
\begin{thm} \label{thm:HA}\cite{ebrahimipatras_15}
The graded algebra $\overline T(T(\mathcal A))$ equipped with the coproduct \eqref{HopfAlg} is a graded connected non-commutative and non-cocommutative Hopf algebra. 
\end{thm}

The crucial observation is that the coproduct \eqref{HopfAlg} can be split into two parts as follows. On $T(\mathcal A)$ define the {\it{left half-coproduct}} by
\begin{equation}
\label{HAprec+}
	\Delta^+_{\prec}(a_1 \cdots a_n) := \sum_{1 \in S \subseteq [n]} a_S \otimes a_{J^S_{[n]}},
\end{equation}
and
\begin{equation}
\label{HAprec}
	\Delta_{\prec}(a_1 \cdots a_n) := \Delta^+_{\prec}(a_1 \cdots a_n) - a_1 \cdots a_n \otimes \un. 
\end{equation}
The {\it{right half-coproduct}} is defined by
\begin{equation}
\label{HAsucc+}
	\Delta^+_{\succ}(a_1 \cdots a_n) := \sum_{1 \notin S \subset [n]} a_S \otimes a_{J^S_{[n]}}
\end{equation}
and
\begin{equation}
\label{HAsucc}
	\Delta_{\succ}(a_1 \cdots a_n) := \Delta^+_{\succ}(a_1 \cdots a_n) -  \un \otimes a_1 \cdots a_n.
\end{equation}
Which yields $\Delta = \Delta^+_{\prec} + \Delta^+_{\succ}$, and $\Delta(w) = \Delta_{\prec}(w) + \Delta_{\succ}(w) + w \otimes \un + \un \otimes w.$ This is extended to $T(T(\mathcal A))$ by defining
\begin{eqnarray*}
	\Delta^+_{\prec}(w_1 | \cdots | w_m) &:=& \Delta^+_{\prec}(w_1)\Delta(w_2) \cdots \Delta(w_m) \\
	\Delta^+_{\succ}(w_1 | \cdots | w_m) &:=& \Delta^+_{\succ}(w_1)\Delta(w_2) \cdots \Delta(w_m). 
\end{eqnarray*}

\begin{thm}  \cite{ebrahimipatras_15} \label{thm:bialg}
The bialgebra $\overline T(T(\mathcal A))$ equipped with $\Delta_{\succ}$ and $\Delta_{\prec}$ is an unshuffle bialgebra. 
\end{thm}

\begin{rmk}
The (graded) dual of $\overline T(T(\mathcal A))$ is equipped with what Gerstenhaber and Voronov called a left increasing product, that means, that given $x,y$ in $\overline T(T(\mathcal A))^\ast$, $deg_T(x\shuffle y)\geq deg_T(x)$, where we write $\shuffle$ for the product dual to $\Delta$ and where $deg_T(w_1|\cdots |w_n):=n$. It follows then from \cite[Lemma 8]{gv} that $\overline T(T(\mathcal A))^\ast$ is naturally a brace algebra.
\end{rmk}


\section{Shuffle groups, exponentials and logarithms}
\label{sect:group}

Let us denote from now on $\overline B = k \oplus B$, where $B=\bigoplus_{n \geq 1}B_n$ is a graded connected unshuffle bialgebra with structure maps $m_B=\cdot_B$, $\Delta$, $\Delta_\prec$, $\Delta_\succ$ and augmentation map $e : \overline B \to k$. The latter acts as the identity on $k$ and as the null map on $B$. For notational simplicity we do not distinguish between $e$ as a map to $k$ and $e$ viewed as a linear endomorphism of $\overline B=k \oplus B$. The space $\End(\overline B)$ of linear endomorphisms of $\overline B$ is a $k$-algebra with respect to the convolution product defined for $f, g \in \End(\overline B)$, by $f * g := m_{B} \circ (f \otimes g) \circ \Delta.$ Its unit is the augmentation map $e$; the antipode is the convolution inverse of $\id$, the identity map of $\overline B$:
\begin{equation}
\label{antipode}
	S=\frac{e}{e + P}=\sum_{i \ge 0} (-1)^i P^{* i},
\end{equation}
where $P:= \id - e$ is the projector that acts as the identity on $B$ and as the null map on $k$. 

The set $\mathrm{Lin}(\overline{B},{A})$ of linear maps from $\overline{B}$ to a unital commutative algebra $A$ over $k$ is also an algebra with respect to the convolution product defined for $f,g \in \mathrm{Lin}(\overline{B},{A})$ by
\begin{equation}
\label{convProd}
	f * g := m_{A}\circ (f\otimes g)\circ \Delta,
\end{equation} 
where $m_{A}$ stands for the product map in ${A}$. The augmentation map $e$ is the unit for $\ast$. We define accordingly the {\it{left}} and {\it{right half-convolution}} products on $\mathrm{Lin}(B,{A})$:
$$
	f\prec g:=m_{A}\circ (f\otimes g)\circ \Delta_\prec 
	\quad\
	{\rm{and}}
	\quad\
	f\succ g:=m_{A}\circ (f\otimes g)\circ \Delta_\succ,
$$
which split the associative convolution product $f * g = f \succ g + f \prec g$. These operations satisfy the  {\it{shuffle identities}}:
\begin{eqnarray}
	(f \prec g)\prec h &=& f \prec(g * h)        		\label{A1}\\
  	(f \succ g)\prec h &=& f \succ(g\prec h)   		\label{A2}\\
   	f \succ(g\succ h)   &=& (f  * g)\succ h,	      	\label{A3}
\end{eqnarray}
where $f,g,h\in \mathrm{Lin}(B,{A})$. Notice that these identities (also known as dendriform relations) follow from the dual ones defining unshuffle coalgebras in Definition \ref{def:unshufflecoalg}. Relations \eqref{A1}-\eqref{A3} are extended by using Sch\"utzenberger's trick, that is, by setting $e\succ f:=f$, $f \prec e:=f$, $e \prec f:=0$ and $f\succ e:=0$ for $f \in \mathrm{Lin}(B,{A})$, turning $e$ into a unit (the notion dual to the one of unshuffle coalgebra counit --  notice, however, that $e\prec e$ and $e\succ e$ are left undefined, whereas $e\ast e=e$). Identities \eqref{A1}-\eqref{A3} (together with the ones satisfied by $e$) define a {\it{(unital) shuffle algebra}}. 

Recall that a (left) pre-Lie algebra \cite{Cartier11,manchon_11} is a $k$-vector space $V$ equipped with a bilinear product $\rhd$ such that
$$
	(x\rhd y)\rhd z-x\rhd (y\rhd z)=(y\rhd x)\rhd z-y\rhd (x\rhd z).
$$
The notion of right pre-Lie algebra is defined analogously. The bracket $[x,y]:=x\rhd y - y\rhd x$ satisfies the Jacobi identity and defines a Lie algebra structure on $V$.

\begin{prop}  
The space ${\mathcal L}_B({A}):=\mathrm{Lin}(B,{A})$ (respectively~$\overline{\mathcal L}_B({A}):=\mathrm{Lin}(\overline B,{A})$) equipped with the half-shuffles $(\prec, \succ)$ is a shuffle algebra (respectively~unital shuffle algebra with unit $e$). Moreover, the space ${\mathcal L}_B({A})$ is naturally equipped with a (left) pre-Lie algebra structure given by
\begin{equation}
\label{preLie}
	f \rhd g := f \succ g  - g \prec f.
\end{equation}
\end{prop}

The last assertion of the proposition is a general property of shuffle algebras. Its verification is left to the reader.

Recall that, given a commutative $k$-algebra $A$, a character $\Phi \in \mathrm{Lin}(\overline B,{A})$ is a unital and multiplicative map, i.e. $\Phi(\un)=1_{A}$, and for $a,b \in B$, $\Phi(a\cdot_B b)=\Phi(a)\Phi(b).$ Since $\overline B$ is a Hopf algebra, the set $G_B({A}) \subset \mathrm{Lin}(\overline B,{A})$ of characters (respectively~$G_B(\cdot)$) is a group and a pro-algebraic group with respect to the convolution product (respectively~a group and pro-algebraic group-valued functor). The inverse of $\Phi \in G_B({A})$ is $\Phi^{-1}:=\Phi \circ S$. 

\begin{defn}
A shuffle group is a functor $\mathcal S$ from commutative unital algebras to pro-algebraic groups represented by a graded connected shuffle bialgebra $\overline B$:
$$
	\mathcal{S}({A})=G_B({A})=\mathrm{Alg}(\overline B, {A}).
$$
\end{defn}

As a group-valued functor $\mathcal{S}$ is nothing but the affine pro-algebraic group associated to the bialgebra $\overline B$. The key point for our later developments, and the justification for the introduction of shuffle groups is, that this functor is equipped with natural transformations that do not exist when dealing with classical (pro-)algebraic groups, making its structure much richer than the usual one of (pro-)algebraic groups.

Recall now that an infinitesimal character $\kappa \in \mathrm{Lin}(\overline B,{A})$ is a map such that  $\kappa(\un)=0$ and $\kappa(a \cdot_B b)=0,$ for $a,b\in B$. The space $g_B({A}) \subset \mathrm{Lin}(\overline B,{A})$ of infinitesimal characters is a Lie algebra for the Lie bracket $[f,g]:=f \ast g - g \ast f$. Let us write $\mathrm{Inf}(\overline B,{A})$ for the Lie algebra of infinitesimal characters from $\overline B$ to $A$.

\begin{defn}
A shuffle Lie algebra is a functor $\mathcal S$ from commutative unital algebras to Lie algebras, such that there exists a graded connected shuffle bialgebra $\overline B$ with
$$
	\mathcal{S}({A})=g_B({A})=\mathrm{Inf}(\overline B, {A}).
$$
\end{defn}

Let us fix again $B$ and consider the shuffle group $G_B({A})$ and the shuffle Lie algebra $g_B({A})$. The exponential and logarithm maps are defined for any $\alpha \in \mathrm{Lin}(\overline B,{A})$ vanishing on $k \subset B$ by
\begin{equation}
\label{ExpLog}
	\exp^*\!(\alpha):=e + \sum_{n > 0} \frac{\alpha^{* n}}{n!}  
	\qquad\ 
	\log^*(e+\alpha):=-\sum_{n>0}(-1)^n\frac{\alpha^{*n}}{n}. 
\end{equation}
Both series reduce to finite sums when applied to an element in $\overline B$. 

\begin{thm}[First exponential isomorphism]
The logarithm and exponential maps are set isomorphisms between the group $G_B({A})$ and the Lie algebra $g_B({A})$.
\end{thm}

The Theorem is well-known and holds in general for characters and infinitesimal characters on any graded connected bialgebra. It follows from $\exp(x+y)=\exp(x)\exp(y)$ and $\log(ab)=\log(a)+\log(b)$ when $x$ and $y$ (respectively~$a$ and $b$) commute, as well as from the fact that in a bialgebra the coproduct is a morphism of algebras, together with the observation that in $\overline B\otimes \overline B$ two elements $x\otimes \mathbf{1}$ and $\mathbf{1} \otimes y$ always commute. We refer, e.g., to \cite{egp} for details.

\smallskip 

The following theorem extends results obtained in \cite{ebrahimipatras_15} when dealing with $T(T(\mathcal A))$. Let us start with definitions. For $\alpha \in g_B({A})$ the {\it{left}} and {\it{right half-shuffle}}, or ``time-ordered'', exponentials are defined by
$$
	\mathcal{E}_\prec(\alpha) := \exp^{\prec}(\alpha) :=e + \sum_{n > 0} \alpha^{\prec{n}} 
	\qquad\
	\mathcal{E}_\succ(\alpha) := \exp^{\succ}(\alpha):=e + \sum_{n > 0}  \alpha^{\succ n}, 
$$
where $\alpha^{\prec{n}} := \alpha \prec(\alpha^{\prec{n-1}})$, $\alpha^{\prec{0}}:=e$ (analogously for $\alpha^{\succ n}$). They satisfy by definition the fixed point equations
\begin{equation}
\label{recursion}
	\mathcal{E}_\prec(\alpha)=e + \alpha \prec \mathcal{E}_\prec(\alpha) 
	\qquad 
	\mathcal{E}_\succ(\alpha)= e + \mathcal{E}_\succ(\alpha) \succ \alpha.
\end{equation}
Note that both $\mathcal{E}_\prec(\alpha)$ and $\mathcal{E}_\succ(\alpha)$ reduce to finite sums when applied to an element of $\overline B$ due to $\alpha \in g_B({A})$.  

\begin{lem}\label{inverse-shuffle}
For $\alpha \in g_B({A})$, we have $\mathcal{E}_{\succ}(-\alpha) \ast \mathcal{E}_{\prec}(\alpha) = e,$
and therefore $\mathcal{E}_\succ(-\alpha)=\mathcal{E}^{-1}_\prec(\alpha)$.
\end{lem}

\begin{proof} We follow \cite{ebrahimipatras_15} and see that
\begin{eqnarray*}
	\mathcal{E}_{\succ}(-\alpha) \ast \mathcal{E}_{\prec}(\alpha) - e
	&=& \sum\limits_{n+m\geq 1}(-1)^n\big\{(\alpha^{\succ n})\prec (\alpha^{\prec m}) 
					+ (\alpha^{\succ n})\succ (\alpha^{\prec m})\big\}\\
	&=& \sum\limits_{n>0,m\geq 0}(-1)^n(\alpha^{\succ n})\prec (\alpha^{\prec m}) 
				+ \sum\limits_{n\geq 0,m> 0}(-1)^n(\alpha^{\succ n})\succ (\alpha^{\prec m}).
\end{eqnarray*}
Now, since $(-1)^n(\alpha^{\succ n})\prec (\alpha^{\prec m})=(-1)^n((\alpha^{\succ n-1})\succ \alpha)\prec (\alpha^{\prec m})=(-1)^n(\alpha^{\succ n-1})\succ (\alpha^{\prec m+1})$, the proof follows.
\end{proof}

Another useful result follows from the computation of the composition inverse of the left half-shuffle exponential.

\begin{lem}\label{inverse} 
For $\alpha \in g_B({A})$ and $X := e +Y :=\mathcal{E}_\prec (\alpha)$, then
$$
	\alpha =Y\prec \Big(\sum\limits_{n\geq 0}(-1)^nY^{\ast n}\Big).
$$
\end{lem}

\begin{proof}
We follow \cite{fpshuffle}. From $X=e + \sum_{n > 0}\alpha^{\prec n}$, we get $X-e=Y=\alpha \prec X$. On the other hand, the (formal) inverse of $X$ {\it{for the $\ast$ product}} is given by $X^{-1}=\frac{e}{e+Y}=\sum_{k\geq 0}(-1)^kY^{\ast k}$. We finally obtain 
$$
	\alpha = \alpha\prec e
	  = \alpha\prec (X\ast X^{ -1})
	  = (\alpha\prec X)\prec X^{-1}
	  =Y\prec \Big(\sum\limits_{n\geq 0}(-1)^nY^{\ast n}\Big).
$$
\end{proof}

\begin{thm}[Left and right exponential isomorphisms]\label{thm:Gg}
The left and right half-shuffle exponentials provide natural bijections between $G_B({A})$ and $g_B({A})$: let $\Phi \in G_B({A})$ then there exists a unique $\kappa \in g_B({A})$ such that $\Phi = \mathcal{E}_\prec(\kappa)$, i.e., $\Phi=e + \kappa \prec \Phi.$ Conversely, for $\kappa \in g_B({A})$ there exists a unique character $\mathcal{E}_\prec(\kappa) \in G_B({A})$. Analogous statements hold for the right half-shuffle exponential.
\end{thm}

\begin{proof}
We know from the previous lemma that the implicit equation $\Phi = e + \kappa \prec \Phi = \mathcal{E}_\prec(\kappa)$ has a unique solution $\kappa$ in $\mathrm{Lin}(\overline B,{A})$ with $\kappa(\mathbf{1})=0$. 

Recall that $\overline B$ is canonically isomorphic as an associative algebra to a tensor algebra, $T(I)$ (Chapoton's Theorem \ref{chapo}). Let us consider the infinitesimal character $\mu:=Res(\kappa)$, equal to $\kappa$ on $I$ and to the null map on tensor powers $I^{\otimes n}$, $n \not= 1$. Let us show that $\mu$ also solves the linear fixed point equation $\Phi = e + \mu \prec \Phi$. From this the first part of the theorem follows.

Indeed, for an arbitrary $a\in \overline B=T(I), a=a_1\cdot_B \dots \cdot_B a_n,\ a_i\in I$, notice first that since $\Delta^+_\prec(x\cdot_B y)=\Delta^+_\prec(x)\cdot_B\Delta(y)$, due to the vanishing of $\mu$ on any $T(I)^{\otimes k}$, for $k\not= 1$, we have:
$$
	(\mu\prec \Phi)(a)	=\mu(a_1^{1,\prec})\Phi(a_1^{2,\prec}\cdot_B a_2\cdot_B\cdots\cdot_Ba_n)
				=\kappa(a_1^{1,\prec})\Phi(a_1^{2,\prec}\cdot_B a_2\cdot_B\cdots\cdot_Ba_n).
$$
We used Sweedler's notation $\Delta^+_\prec(x)=x^{1,\prec}\otimes x^{2,\prec}$. From this we immediately obtain 
$$
	\Phi (a)=\Phi(a_1)\Phi(a_2\cdot_B\cdots\cdot_Ba_n)
		=\mu(a_1^{1,\prec})\Phi(a_1^{2,\prec}\cdot_Ba_2\cdot_B\cdots\cdot_Ba_n)
		=(e+\mu\prec \Phi)(a),
$$
since 
$$
	\Phi(a_1)	=(e+\kappa\prec \Phi)(a_1)
			=\kappa(a_1^{1,\prec})\Phi(a_1^{2,\prec})
			=\mu(a_1^{1,\prec})\Phi(a_1^{2,\prec}).
$$ 
Conversely:
$$
	\mathcal{E}_{\prec}(\kappa)(a)
	=(e+\kappa\prec\mathcal{E}_{\prec}(\kappa))(a)
	=\kappa(a_1^{1,\prec})\mathcal{E}_{\prec}(\kappa)(a_1^{2,\prec}
					\cdot_B a_2\cdot_B\cdots\cdot_Ba_n).
$$
Assuming by induction that the property $\mathcal{E}_{\prec}(\kappa)(a_1' \cdot_B \cdots \cdot_B a'_k) = \mathcal{E}_{\prec}(\kappa)(a_1') \cdots\mathcal{E}_{\prec}(\kappa)(a_k')$ holds for elements $a_1' \cdot_B \cdots \cdot_Ba'_k$ in $T(I) = \overline B$ of degree less than the degree of $a$, yields
\begin{eqnarray*}
	\mathcal{E}_{\prec}(\kappa)(a)
	&=&\kappa(a_1^{1,\prec})\mathcal{E}_{\prec}(\kappa)(a_1^{2,\prec})
	\mathcal{E}_{\prec}(\kappa)(a_2) \cdots \mathcal{E}_{\prec}(\kappa)(a_n)\\		
	&=&\mathcal{E}_{\prec}(\kappa)(a_1)\mathcal{E}_{\prec}(\kappa)(a_2) 
	\cdots \mathcal{E}_{\prec}(\kappa)(a_n).
\end{eqnarray*}
\end{proof}

The inverses of the isomorphisms in Theorem \ref{thm:Gg} can be computed explicitly using Lemma \ref{inverse-shuffle} together with the proof of Lemma \ref{inverse}.

\begin{lem}\label{lem:inverse}Let $\Phi\in G_B({A})$, the {\rm{left}} and {\rm{right half-shuffle logarithms}} of $\Phi$ are defined respectively by
\begin{align}
	 \log^{\prec}(\Phi) &:=(\Phi - e)\prec \Phi^{-1}   	\label{leftlog}\\
	 \log^{\succ}(\Phi) &:= \Phi^{-1}\succ (\Phi - e). 	\label{rightlog}
\end{align}
These maps are the compositional inverses to the left respectively right half-shuffle exponentials, i.e., let $\alpha \in g_B({A})$, then 
\begin{align}
	 \log^{\prec}(\mathcal{E}_\prec(\alpha))   &=\alpha, \label{leftloga}\\
	 \log^{\succ}(\mathcal{E}_\succ(\alpha )) &=\alpha . \label{rightloga}
\end{align}
\end{lem}

The next two theorems show how these results translate into fundamental properties of monotone, boolean and free cumulants \cite{ebrahimipatras_17}. Let $(\mathcal A,\varphi)$ be a non-commutative probability space with unital map $\varphi: \mathcal A \to k$. Let $\Phi$ be the extension of $\varphi$ as a character over $\overline H:=\overline T(T(\mathcal A))$, i.e., $\Phi(w_1| \cdots |w_k) := \phi(w_1) \cdots \varphi(w_k)$ ($\varphi$ is first extended to a linear map from $T(\mathcal A)$ to $k$ by $\phi(a_1 \cdots a_n):=\varphi(a_1\cdot_\mathcal A \cdots \cdot_\mathcal A a_n)$). Recall that for $w=a_1 \cdots a_n \in T(\mathcal A)$, the $n$th order multivariate moment is defined by $m_n(a_1,\ldots,a_n):=\phi(w)$. 

\begin{thm} \label{tim:monotonefreeboolean} \cite{ebrahimipatras_17}
Let $\rho: \overline T(T(\mathcal A)) \to k$ be the infinitesimal character defined  by $\rho := \log^*(\Phi) \in g_H(k)$. For $a_1,\dots,a_n\in \mathcal A$, we define $h_n(a_1,\ldots,a_n):=\rho(w)$. Then the $h_n(a_1,\ldots,a_n)$ identify with the multivariate monotone cumulants of \cite{hasebesaigo_11}.
\end{thm}

Consider similarly the infinitesimal characters $\kappa:=\log^\prec (\Phi)$ and $\beta:=\log^\succ(\Phi)$.

\begin{thm}\label{cor:freeboolean} \cite{ebrahimipatras_15,ebrahimipatras_17}
For $w=a_1 \cdots a_n \in T(\mathcal A)$ we set $r_n(a_1, \ldots, a_n ):=\beta(w)$ and $k_n(a_1, \ldots ,a_n ):=\kappa(w)$. Then the $r_n$ and $k_n$ identify respectively with multivariate free and boolean cumulants. 
\end{thm}


\section{Shuffle Lie algebra automorphisms}
\label{sect:automorphisms}

The three bijections $\exp^\ast$, $\mathcal{E}_\prec$ and $\mathcal{E}_\succ$ between $g_B({A})$ and $G_B({A})$ induce six non-trivial (set) automorphisms of $g_B({A})$, namely $\log^\prec\circ \exp^\ast$, $\log^\succ\circ\exp^\ast$, $\log^\prec \circ\ \mathcal{E}_\succ$, and their inverses. We investigate in this section the first two bijections that happen to be related to fundamental operations in the theory of pre-Lie algebras, namely the Magnus expansion and its inverse. We refer to \cite{ebrahimimanchon_09,manchon_11} for details and further references on pre-Lie algebras. The third bijection and its inverse will be studied in the following section. Notice that one could also study the six (set) automorphisms of $G_B({A})$, $ \exp^\ast\circ \log^\prec$, $\exp^\ast\circ\log^\succ$, $ \mathcal{E}_\succ\circ\log^\prec$, and their inverses. The map $ \mathcal{E}_\prec\circ\log^\succ$ is closely related to the Bercovici--Pata bijection, this will be the subject of Section \ref{sect:BPbijection}.

\smallskip

Let us compute $\kappa = \log^\prec(\exp^*\!(\rho))$. An analogous argument provides $\beta= \log^\succ(\exp^*\!(\rho))$. Notice first that
\begin{align*}
	\lefteqn{\frac{d}{dt}\Big(\exp^*\!{(t \rho)} \succ \big(\exp^*\!{((1 - t) \rho)} - e\big) \Big)}\\
		&= (\exp^*\!{(t\rho)}* \rho) \succ \big(\exp^*\!{(\left( 1 - t \right)\rho}) - e\big) 
			 - \exp^*\!{(t\rho)} \succ \big(  \rho *  \exp^*\!{(\left( 1 - t \right) \rho)}\big) \\
		&= \exp^*\!{(t\rho)} \succ \big(\rho \succ \left(\exp^*\!{(\left( 1 - t \right)\rho}) - e\right)\big)
		- \exp^*\!{(t\rho)} \succ \big(  \rho \succ ( \exp^*\!{(\left( 1 - t \right) \rho)- e)} 
		+ \rho \prec  \exp^*\!{(\left( 1 - t \right) \rho)} \big) \\	 
		&= - \exp^*\!{(t\rho)} \succ \rho \prec \exp^*\!{(\left( 1 - t \right) \rho)},  
\end{align*} 
where we applied the shuffle rules \eqref{A1}-\eqref{A3}. Integrating yields
\begin{align*}
	\exp^*\!(\rho) - e 
	&= \int_0^1 \exp^*\!{(s\rho)} \succ \rho \prec \exp^*\!{(\left(1-s \right) \rho)}ds.
\end{align*} 
Again, from the shuffle rules \eqref{A1}-\eqref{A3} we find that
\begin{align}
	\exp^*\!(\rho) - e
	&= \int_0^1 \exp^*\!{(s\rho)} \succ \rho \prec \exp^*\!{(\left(1-s \right) \rho)}ds \nonumber \\
	&= \int_0^1 \left( \exp^*\!{(s\rho)} \succ \rho \prec 
		\exp^*\!{(- s\rho)} \right) \prec \exp^*\!(\rho) ds \nonumber \\
	&= \Big(\int_0^1 \mathrm{e}^{sL_{\rho \succ}} \mathrm{e}^{- sR_{\prec \rho}} 
		\left( \rho \right) ds\Big) \prec \exp^*\!(\rho) \nonumber \\
	&= \Big(\int_0^1 \mathrm{e}^{sL_{\rho \rhd}} (\rho) ds\Big) \prec \exp^*\!(\rho) 
	=\frac{\mathrm{e}^{L_{\rho \rhd}} - \id}{L_{\rho \rhd}} (\rho) \prec \exp^*\!(\rho), \label{invMagnus1}
\end{align} 
where the left and right multiplication maps $L_{\alpha \succ}$ and $R_{ \prec \alpha}$ are defined respectively by $L_{\alpha \succ}(\beta):=\alpha \succ \beta$ and $R_{\prec \alpha}(\beta):=\beta \prec \alpha$. Note that due to \eqref{A2} we have  $L_{\beta \succ} \circ R_{\prec \alpha} = R_{\prec \alpha} \circ L_{\beta \succ}$, and that $(L_{\alpha \succ} - R_{\prec \alpha})(\beta)= \alpha \succ \beta - \beta \prec \alpha$ gives the pre-Lie product $L_{\alpha \rhd}(\beta):=\alpha \rhd \beta$ defined in \eqref{preLie}. This implies that $\mathrm{e}^{sL_{\rho \succ}} \mathrm{e}^{- sR_{\prec \rho}}=\mathrm{e}^{sL_{\rho \rhd}}$. Recall Lemma \ref{lem:inverse}, and the fact that $\exp^*\!(\rho)=\mathcal{E}_\prec(\kappa)$ solves uniquely $\Phi = e + \kappa \prec \Phi$. From \eqref{invMagnus1} we have
$$
	\exp^*\!(\rho)  = e + \Big(\frac{\mathrm{e}^{L_{\rho \rhd}} 
	- \id}{L_{\rho \rhd}} (\rho)\Big) \prec \exp^*\!(\rho).
$$

\begin{thm} The infinitesimal character $\kappa := \log^{\prec}(\exp^*\!(\rho)) \in g_B({A})$ is given by
\begin{equation}
\label{invMagnus2}
	\kappa = \frac{\mathrm{e}^{L_{\rho \rhd}} - \id}{L_{\rho \rhd}} (\rho)=:W(\rho). 
\end{equation}
Conversely
\begin{equation}
\label{Magnus1}
	\rho=\frac{L_{\rho \rhd}}{\mathrm{e}^{L_{\rho \rhd}}- \id}(\kappa)=:\Omega'(\kappa).
\end{equation}	

Similarly, for $\beta :=  \log^{\succ}(\exp^*\!(\rho))  \in G_B({A})$ we have
\begin{equation}
\label{Magnus3}
	\beta = \frac{\mathrm{e}^{-L_{\rho \rhd}} - \id}{L_{\rho \rhd}} (-\rho) 
		  =-W(-\rho).
\end{equation}	
Finally, from $\Phi=\exp^*\!(-\Omega'(-\beta)) = \exp^*\!(\Omega'(\kappa))$ we deduce relations between the infinitesimal characters $\kappa$ and $\beta$
\begin{align}
\label{Magnus4}
	\kappa&=W\big(-\Omega'(-\beta)\big) \\
\label{Magnus5}
	\beta&=-W\big(-\Omega'(\kappa)\big).
\end{align}
\end{thm}
These results allow, in the particular case where $B=H:=T(T(\mathcal A))$, to compare the three notions of free, boolean and monotone cumulants \cite{ebrahimipatras_15}. The automorphism $\Omega'$ of the set of infinitesimal characters  is called pre-Lie Magnus expansion \cite{chapoton_09, cp, ebrahimimanchon_09}. It is given by an expansion in terms of the Bernoulli numbers $B_m$
\begin{equation}
\label{preLieMagnus}
	\Omega'(\alpha) = \sum\limits_{m\ge 0} \frac{B_m}{m!}\ L^{(m)}_{\Omega'(\alpha)  \rhd}(\alpha)
            =\alpha - \frac{1}{2}\alpha \rhd \alpha + \sum\limits_{m\ge 2} \frac{B_m}{m!}\ L^{(m)}_{\Omega'(\alpha)  \rhd}(\alpha).
\end{equation}
Here $L^{(m)}_{a \rhd}(b):=L^{(m-1)}_{a \rhd}(a \rhd b)$, $L^{(0)}_{a \rhd}=\id$. The inverse of $\Omega'$ is the map $W$ which expands as a pre-Lie exponential  
\begin{equation}
\label{eq:W1}
	W(\alpha) = \alpha + \frac 12 \alpha\rhd\alpha + \frac 16 \alpha\rhd(\alpha\rhd \alpha) + \cdots.
\end{equation}


\section{Shuffle adjoint actions}
\label{sect:shuffleadjoint}

From \eqref{recursion} and the identity $\mathcal{E}_\prec(\kappa)=\Phi=\mathcal{E}_\succ(\beta)$ we derive the following relations
\begin{equation}
\label{booleanfree}
	\beta = \Phi^{-1} \succ \kappa \prec \Phi
\end{equation}
and 
\begin{equation}
\label{freeboolean}
\kappa = \Phi \succ \beta  \prec \Phi^{-1}. 
\end{equation}
They follow from the shuffle rules \eqref{A1}-\eqref{A3} together with $\Phi \in G_B({A})$. Indeed, from $\Phi=\exp^*\!(-\Omega'(-\beta))$, and $\beta=-W\big(-\Omega'(\kappa)\big)$ together with \eqref{eq:W1} it follows that 
\begin{align*}				
	 \beta &= \mathrm{e}^{L_{\Omega'(-\beta)\rhd}}\big(W(-\Omega'(-\beta)\big)  \\
		 &=\mathrm{e}^{L_{\Omega'(-\beta) \succ}} \mathrm{e}^{-R_{\prec \Omega'(-\beta)}}\big(W(-\Omega'(-\beta)\big) \\
		 &= \exp^*\!(\Omega'(-\beta)) \succ W\big(-\Omega'(-\beta)\big)\prec \exp^*\!\big(-\Omega'(-\beta)\big)\\
		 &= \Phi^{-1} \succ \kappa \prec \Phi.
\end{align*}

Here we used that $\kappa=W\big(-\Omega'(-\beta)\big)$. This leads to the general definition of shuffle adjoint action.

\begin{defn}
For $\Phi\in G_B({A})$ and $\mu\in g_B({A})$, the {\rm{shuffle adjoint action}} of the group $G_B({A})$ on the Lie algebra $g_B({A})$ is defined  by
$$
	Ad_\Phi(\mu):=\Phi^{-1} \succ \mu \prec \Phi.
$$
\end{defn}

Note that classical adjunction (by conjugacy) would read $Ad^*_\Phi(\mu):=\Phi^{-1} * \mu * \Phi.$ The formula
$$
	\mathrm{e}^{L_{\Omega'(-\beta)\rhd}}(\mu ) 
	=\exp^*\!(\Omega'(-\beta)) \succ \mu \prec \exp^*\!\big(-\Omega'(-\beta)\big) 
	= \Phi^{-1} \succ \mu \prec \Phi
$$
insures that the action of $Ad_\Phi$ is indeed mapping $g_B({A})$ to itself. Similarly, at the Lie algebra level we get
		
\begin{defn}\label{def:groupaction}
For $\gamma_1,\gamma_2 \in g_B({A})$, the three shuffle adjoint actions of the  Lie algebra $g_B({A})$ on itself are defined by
\begin{equation}
\label{shuffleaction}
	  {ad}^\ast_{\gamma_1}(\gamma_2)
	  := \exp^\ast(-\gamma_1) \succ \gamma_2 \prec  \exp^\ast(\gamma_1)
\end{equation}
\begin{equation}
\label{shuffleaction1}
	  {ad}^\prec_{\gamma_1}(\gamma_2)
	  := \mathcal{E}^{-1}_{\prec}(\gamma_1) \succ \gamma_2 
	  \prec  \mathcal{E}_{\prec}(\gamma_1)
\end{equation}
\begin{equation}
\label{shuffleaction2}
	  {ad}^\succ_{\gamma_1}(\gamma_2)
	  := \mathcal{E}_{\succ}(\gamma_1) \succ \gamma_2 
	  \prec  \mathcal{E}^{-1}_{\succ}(\gamma_1). 
\end{equation}
For later use, we also introduce the notation $\gamma_2^{\gamma_1}:={ad}^\prec_{\gamma_1}(\gamma_2)$.
\end{defn}

\begin{lem}
We have 
$$
	{ad}^\succ_{\gamma_1}(\gamma_2)={ad}^\prec_{-\gamma_1}(\gamma_2).
$$
\end{lem}

\begin{proof}
Indeed, from $\mathcal{E}_{\succ}(\gamma_1) = \exp^*\!\big(-\Omega'(-\gamma_1)\big)$ we find that 
\begin{align*}
	 {ad}^\succ_{\gamma_1}(\gamma_2)	
	 			&= \mathcal{E}_{\succ}(\gamma_1) 
					\succ \gamma_2 \prec  \mathcal{E}^{-1}_{\succ}(\gamma_1)\\
	  			&= \exp^*\!\big(-\Omega'(-\gamma_1)\big)\succ \gamma_2 \prec 
					\exp^*\!\big(\Omega'(-\gamma_1)\big)\\
				&= \mathcal{E}^{-1}_{\prec}(-\gamma_1)  \succ \gamma_2 
					\prec \mathcal{E}_{\prec}(-\gamma_1)\\
				&= {ad}^\prec_{-\gamma_1}(\gamma_2).
\end{align*}
\end{proof}

As an example we consider the following proposition, which provides a formula for $\gamma_2^{\gamma_1}$ in the context of $B=\overline{H}=\overline{T}(T(\mathcal A))$ and $A=k$. Recall that this is the setting relevant to applications in non-commutative probablility.

\begin{prop}\label{prop:NicaLemma3.2}
Let $\gamma_1$, $\gamma_2$ be infinitesimal characters in $g_B(k)$, where  $B=\overline{T}(T(\mathcal A))$. For a word $w=a_1 \cdots a_n \in T(\mathcal A)$ we find 
\begin{equation}
	\gamma_2^{\gamma_1}(w) = \sum_{S \subset [n] \atop 1,n \in S} \gamma_2(a_S) \mathcal{E}_\prec(\gamma_1)(a_{J^S_{[n]}}).  
\end{equation}
\end{prop}

\begin{proof}
First we remark that this result coincides with \cite[Lemma 3.2]{nica_09}. Our proof follows the argument in \cite[Theorem 14]{ebrahimipatras_15}. From $\gamma_2^{\gamma_1}=\mathcal{E}^{-1}_{\prec}(\gamma_1) \succ \gamma_2 \prec  \mathcal{E}_{\prec}(\gamma_1)$ we deduce that $\mathcal{E}_{\prec}(\gamma_1) \succ \gamma_2^{\gamma_1} = \gamma_2 \prec  \mathcal{E}_{\prec}(\gamma_1)$, such that  
\begin{align}
	\mathcal{E}_{\prec}(\gamma_1) \succ \gamma_2^{\gamma_1} (w) 
	&= \gamma_2^{\gamma_1} (w) 
	+ \sum_{j=1}^{n-1} \mathcal{E}_{\prec}(\gamma_1)(a_{j+1} \cdots a_n)\gamma_2^{\gamma_1} (a_1 \cdots a_j)
	=\sum_{1 \in S \subseteq [n]} \gamma_2(a_S) \mathcal{E}_{\prec}(\gamma_1)(a_{J^S_{[n]}}).
\end{align}
From this we obtain 
\begin{align}
	 \gamma_2^{\gamma_1} (w) 
	 	&=\sum_{1 \in S \subseteq [n]} \gamma_2(a_S) \mathcal{E}_{\prec}(\gamma_1)(a_{J^S_{[n]}})
		- \sum_{j=1}^{n-1} \mathcal{E}_{\prec}(\gamma_1)(a_{j+1} \cdots a_n)\gamma_2^{\gamma_1}(a_1 \cdots a_j)\\
		&= \sum_{1,n \in S \subseteq [n]}\gamma_2(a_S) \mathcal{E}_{\prec}(\gamma_1)(a_{J^S_{[n]}})
		 + \sum_{1 \in S \subset [n] \atop n \notin S}\gamma_2(a_S) \mathcal{E}_{\prec}(\gamma_1)(a_{J^S_{[n]}})
		 - \sum_{j=1}^{n-1}  \mathcal{E}_{\prec}(\gamma_1)(a_{j+1} \cdots a_n)\gamma_2^{\gamma_1}(a_1 \cdots a_j) \label{calc1}.
\end{align}
Observe that $\gamma_2(a)=\gamma_2^{\gamma_1}(a)$ for a letter $a \in T_1(\mathcal A)$. For a word of length $n=2$ we find 
$$
	\gamma_2^{\gamma_1}(a_1a_2) 
	=\gamma_2(a_1a_2) + \gamma_2(a_1)\gamma_1(a_2) -  \mathcal{E}_{\prec}(\gamma_1)(a_2)\gamma_2^{\gamma_1}(a_1) 
	= \gamma_2(a_1a_2). 
$$
For $n>2$ we use induction and write $\gamma_2^{\gamma_1}(a_1 \cdots a_j)= \sum_{1,j \in S \subset [j]} \gamma_2(a_S)  \mathcal{E}_{\prec}(\gamma_1)(a_{J^S_{[j]}})$ in \eqref{calc1}. Then
\begin{align}
	\gamma_2^{\gamma_1}(w) &= \sum_{1,n \in S \subseteq [n]} \gamma_2(a_S) \mathcal{E}_{\prec}(\gamma_1)(a_{J^S_{[n]}})
		 + \sum_{1 \in S \subset [n] \atop n \notin S}\gamma_2(a_S) \mathcal{E}_{\prec}(\gamma_1)(a_{J^S_{[n]}})\\
		 &- \sum_{j=1}^{n-1}  \mathcal{E}_{\prec}(\gamma_1)(a_{j+1} \cdots a_n) \sum_{1,j \in T \subset [j]} \gamma_2(a_T)  \mathcal{E}_{\prec}(\gamma_1)(a_{J^T_{[j]}})\\
		&=\sum_{1,n \in S \subseteq [n]} \gamma_2(a_S)  \mathcal{E}_{\prec}(\gamma_1)(a_{J^S_{[n]}}).
\end{align}
In the last step we used that 
$$
	0= \sum_{1 \in S \subset [n] \atop n \notin S}\gamma_2(a_S)  \mathcal{E}_{\prec}(\gamma_1)(a_{J^S_{[n]}})\\
		 - \sum_{j=1}^{n-1} \Big(\sum_{1,j \in T \subset [j]}\gamma_2(a_T)  \mathcal{E}_{\prec}(\gamma_1)(a_{J^T_{[j]}})\Big) \mathcal{E}_{\prec}(\gamma_1)(a_{j+1} \cdots a_n).
$$
\end{proof}

\begin{rmk} Note that for $B=\overline{T}(T(\mathcal A))$ the character $\mathcal{E}_\prec(\gamma_2^{\gamma_1}) \in G_B(k)$ satisfies the equation
\begin{align}
	\mathcal{E}_\prec(\gamma_2^{\gamma_1}) 
	= e + \mathcal{E}^{-1}_\prec({\gamma_1}) \succ \gamma_2\prec \mathcal{E}_\prec(\gamma_2+{\gamma_1}),
	\label{NicaTheorem1.3}
\end{align}
which follows from the  left half-shuffle fixed point equation for $\mathcal{E}_\prec(\gamma_2^{\gamma_1})$. Using the notion of left subordination product (see Definition \ref{def:subordination} below) one can deduce from \eqref{NicaTheorem1.3} Theorem 1.3 in \cite{nica_09}.
\end{rmk}

\begin{prop}\label{prop:compat}
For $\Phi \in G_B({A})$ and $\mu,\nu \in g_B({A})$ we have the following compatibilities of the shuffle group action with the half-shuffle products. 
\begin{align*}
	\Phi^{-1} \succ (\mu \succ \nu ) \prec \Phi 
	&= (\Phi^{-1} \ast \mu \ast \Phi) \succ (\Phi^{-1} \succ \nu \prec \Phi)\\
	\Phi^{-1} \succ (\mu \prec \nu ) \prec \Phi 
	&= (\Phi^{-1} \succ \mu \prec \Phi) \prec (\Phi^{-1} \ast \nu \ast \Phi)
\end{align*}
\end{prop}

\begin{proof}
We show the first identity. 
\begin{align*}
	\Phi^{-1} \succ (\mu \succ \nu ) \prec \Phi 
	&= \Phi^{-1} \succ \big(\mu \succ (\nu  \prec \Phi)\big)\\
	&= (\Phi^{-1} * \mu) \succ (\nu  \prec \Phi)
	= (\Phi^{-1} \ast \mu \ast \Phi) \succ (\Phi^{-1} \succ \nu \prec \Phi).
\end{align*}
The second identity follows by a similar calculation.
\end{proof}

From Proposition \ref{prop:compat} we deduce for $\Phi \in G_B({A})$ and $\mu,\nu \in g_B({A})$ that
$$
	\Phi^{-1} \succ (\mu \rhd \nu ) \prec \Phi 
	= (\Phi^{-1} \ast \mu \ast \Phi) \rhd (\Phi^{-1} \succ \nu \prec \Phi).
$$

\begin{cor}\label{cor:eta-series}
For $\gamma_1 \in g_B({A})$, we deduce from $\mathcal{E}^{-1}_{\prec}(\gamma_1)=\mathcal{E}_{\succ}(-\gamma_1)$ that 
\begin{align*}
	{ad}^\prec_{\gamma_1}(\gamma_1)
				&=  \mathcal{E}^{-1}_{\prec}(\gamma_1) \succ \gamma_1 
					\prec  \mathcal{E}_{\prec}(\gamma_1)\\
				& = \mathcal{E}^{-1}_{\prec}(\gamma_1) 
					\succ (\mathcal{E}_{\prec}(\gamma_1) - e) 
					  =  \log^{\succ}\big(\mathcal{E}_{\prec}(\gamma_1)\big),
\end{align*}
as well as
$$ 
	{ad}^\prec_{\gamma_1}(\gamma_1)= -(\mathcal{E}^{-1}_{\prec}(\gamma_1)  - e) 
					\prec \mathcal{E}_{\prec}(\gamma_1) {ad}^\prec_{\gamma_1}(\gamma_1)= -  \log^{\prec}\big(\mathcal{E}^{-1}_{\prec}(\gamma_1)\big). 
$$
\end{cor}


\section{Universal shuffle group laws}
\label{sec:BCHcumulants}

From the general theory of shuffle algebras, it is known that the Baker--Campbell--Hausdorff (${\mathrm{BCH}}$) formula and the pre-Lie Magnus expansion are closely related. See \cite{cp,ebrahimipatras_14,manchon_11} for more details. In classical Lie theory, the ${\mathrm{BCH}}$ formula allows to define a universal group law at the Lie algebra level. Indeed, consider a complete connected cocommutative Hopf algebra over $k$, then the logarithm and the exponential maps define inverse isomorphisms from the group of group-like elements to the Lie algebra of primitive elements that allow to transport the product of the group to the Lie algebra.

In the setting of unshuffle bialgebras (using our previous notations), the ${\mathrm{BCH}}$ group law on $g_B({A})$ reads for $\gamma_1,\gamma_2 \in g_B({A})$:
\begin{equation}
\label{BCHlaw}
	\gamma_1 \ast_{{\mathrm{BCH}}} \gamma_2 
	={\mathrm{BCH}}(\gamma_1,\gamma_2)
	:= \log^\ast \big(\exp^\ast (\gamma_1)\ast\exp^\ast (\gamma_2)\big).
\end{equation}
The formula has various closed expressions in terms of iterated Lie brackets of $\gamma_1$ and $\gamma_2$, which usually involve descent classes in the symmetric group algebras \cite{reutenauer}. The two other exponential and logarithmic isomorphisms between $G_B({A})$ and $g_B({A})$ do allow for the definition of two other universal group laws on $g_B({A})$, that we name {\it{left}} and {\it{right shuffle group laws}}. 

\begin{defn}[Shuffle group laws] \label{cor:shuffleBCH} For $\gamma_1 ,\gamma_2 \in g_B({A})$, the two group laws on $g_B({A})$ associated to the two half-shuffle isomorphisms between $G_B({A})$ and $g_B({A})$ are defined by:
\begin{equation}
\label{BCHshuffle1}
	\gamma_1 \# \gamma_2:=\log^\prec \big( \mathcal{E}_{\prec}(\gamma_1) * \mathcal{E}_{\prec}(\gamma_2)\big),
\end{equation}
and
\begin{equation}
\label{BCHshuffle2}
	\gamma_1\odot\gamma_2 :=\log^\succ \big(\mathcal{E}_{\succ}(\gamma_1) * \mathcal{E}_{\succ}(\gamma_2)\big).
\end{equation}
\end{defn}

\begin{lem}We have
\begin{equation}
\label{BCHproduct1}
	\gamma_1 \# \gamma_2 
	= W \big(\mathrm{BCH} \big(\Omega'(\gamma_1),\Omega'(\gamma_2)\big)\big),
\end{equation}
\begin{equation}
\label{BCHproduct1bis}
	\gamma_1 \odot \gamma_2 
	= -(-\gamma_2\# -\gamma_1).
\end{equation}
\end{lem}

\begin{proof}The first identity follows from the identities expressing the canonical automorphisms of $g_B({A})$. We verify the last equation:
\begin{eqnarray*}
	\mathcal{E}_{\succ}(\gamma_1) * \mathcal{E}_{\succ}(\gamma_2) 
	&=&  \exp^*\!\big(-\Omega'(-\gamma_1)\big)  * \exp^*\!\big(-\Omega'(-\gamma_2)\big)\\
	&=& \exp^*\!\Big(\mathrm{BCH}\big(-\Omega'(-\gamma_1),-\Omega'(-\gamma_2)\big)\Big)\\
	&=& \exp^*\!\big(-\Omega'(-\gamma_2 \# -\gamma_1)\big)
	=\mathcal{E}_{\succ}\big(-(-\gamma_2 \# -\gamma_1)\big),
\end{eqnarray*}		
where we used that $\mathrm{BCH}(-a,-b)=-\mathrm{BCH}(b,a)$.
\end{proof}

Notice that identity (\ref{BCHproduct1}) rewrites
\begin{equation}
\label{BCHpreLie}
	\mathrm{BCH}\big(\Omega'(\gamma_1),\Omega'(\gamma_2)\big)=\Omega'(\gamma_1 \# \gamma_2).
\end{equation}
It follows that $W(\gamma_1) \# W(\gamma_2)= W\big(\mathrm{BCH} (\gamma_1,\gamma_2)\big)$, and $W(\gamma_1) \# W(\gamma_2) = W(\gamma_1) + \mathrm{e}^{L_{\gamma_1\rhd}}W(\gamma_2)$. The product $\#$ is thus given by the simple formulas in the next lemma.

\begin{lem} For the left shuffle group law we find
\begin{equation}
\label{BCHproduct2}
	\gamma_1 \# \gamma_2 
	= \gamma_1+\mathrm{e}^{L_{\Omega'(\gamma_1) \rhd}}\gamma_2
	=\gamma_1 + {Ad}_{\mathcal{E}^{-1}_{\prec}(\gamma_1)}(\gamma_2).
\end{equation}
\end{lem}

\begin{proof}
Indeed, as mentioned before, from $R_{\prec \gamma_1} \circ L_{\succ \gamma_2}= L_{\succ \gamma_2} \circ R_{\prec \gamma_1}$ we find that 
$$
	\mathrm{e}^{L_{\Omega'(\gamma_1) \rhd}}\gamma_2 = \mathrm{e}^{L_{\Omega'(\gamma_1) \succ}}\mathrm{e}^{-R_{\prec \Omega'(\gamma_1) }}\gamma_2,
$$ 
which implies that
\begin{equation}
\label{key-relation}
	\mathrm{e}^{L_{\Omega'(\gamma_1) \rhd}}\gamma_2 
	= \exp^*\!\big(\Omega'(\gamma_1)\big) \succ \gamma_2 \prec 
	\exp^*\! \big(-\Omega'(\gamma_1)\big).
\end{equation}
Therefore, in terms of the shuffle action given in Definition \ref{def:groupaction}, we see that $\mathrm{e}^{L_{\Omega'(\gamma_1) \rhd}}\gamma_2 = {Ad}_{\mathcal{E}^{-1}_{\prec}(\gamma_1)}(\gamma_2)$, where $\mathcal{E}_{\prec}(\gamma_1)= \exp^*\!\big(\Omega'(\gamma_1))$. This yields the second formula for the product in \eqref{BCHproduct2} 
\begin{equation}
\label{BCHproduc3}
	\gamma_1 \# \gamma_2 = \gamma_1 + {Ad}_{\mathcal{E}^{-1}_{\prec}(\gamma_1)}(\gamma_2).
\end{equation}
\end{proof}


\section{Additive convolutions in free, monotone and boolean probability}
\label{sec:convol}

Let us consider now the dual side of the structures studied in the previous section. Recall once again general ideas underlying Lie theory and consider a complete connected cocommutative Hopf algebra over $k$. Then, since the logarithm and the exponential maps define inverse (set) isomorphisms from the group of group-like elements to the Lie algebra of primitive elements, they also allow (by the process dual to the one used to define $\ast_{BCH}$) to transport the additive group law on the Lie algebra to the group of group-like elements.

In the setting of unshuffle bialgebras (with our previous notations), the additive group law on $G_B({A})$, written $+_{{\mathrm{BCH}}}$, reads:
\begin{equation}
\label{BCHlawbis}
	\exp^\ast(\gamma_1) +_{{\mathrm{BCH}}} \exp^\ast(\gamma_2) :=\exp^\ast(\gamma_1+\gamma_2),
\end{equation}
for $\gamma_1,\gamma_2\in g_B({A})$. This process can be repeated with the two other canonical isomorphisms between $g_B({A})$ and $G_B({A})$, leading to the next definition, where the terminology is motivated by free and boolean probabilities as will become clear later on.

\begin{defn}\label{def:additive1}
The {\rm{left}} and {\rm{right additive convolutions}} are commutative group laws on $G_B({A})$ defined respectively by
\begin{equation}
	\mathcal{E}_{\prec}(\gamma_1)  \ \boxprec \mathcal{E}_{\prec}(\gamma_2)  
	:=\mathcal{E}_{\prec}(\gamma_1 + \gamma_2)
\end{equation}
and 
\begin{equation}
	\mathcal{E}_{\succ}(\gamma_1) \ \boxsucc \mathcal{E}_{\succ}(\gamma_2) 
	:=\mathcal{E}_{\succ}(\gamma_1 + \gamma_2). 
\end{equation}
\end{defn}

By their very definitions, the new products are linearised by the left respectively right half-shuffle logarithms, that is, for $\gamma_1 ,\gamma_2 \in g_B({A})$ 
\begin{equation}
\label{linearisation1}
	\log^{\prec}\big(\mathcal{E}_{\prec}(\gamma_1)  \ \boxprec \mathcal{E}_{\prec}(\gamma_2)\big)
	=\gamma_1 + \gamma_2,
\end{equation}
and
\begin{equation}
\label{linearisation2}
	\log^{\succ}\big(\mathcal{E}_{\succ}(\gamma_1)\ \boxsucc \mathcal{E}_{\succ}(\gamma_2)\big)
	=\gamma_1 + \gamma_2.
\end{equation}

Notice that these properties yield naturally the definition of generalised power maps on $G_B({A})$:

\begin{defn}[Half-shuffle powers]\label{hsp}
For $\gamma \in g_B({A})$ and scalar $s \in k$ we set
$$
	\mathcal{E}_{\prec}(\gamma)^{{\bboxprec}\, s }
	:= \mathcal{E}_{\prec}(s\gamma),
	\qquad\ 
	\mathcal{E}_{\succ}(\gamma)^{{\bboxsucc}\, s }
	:= \mathcal{E}_{\succ}(s\gamma),
$$
so that, for example, for $\Phi\in G_B({A})$, $\Phi^{{\bboxprec}\, s}=\mathcal{E}_{\prec}(s \log^\prec(\Phi))$.
\end{defn}

To investigate further the group laws $\boxprec$ and $\boxsucc$, we recall the key identity
\begin{equation}
\label{transforming}
	\mathcal{E}_{\prec}\big(W(\gamma)\big)	
	= \exp^*\!(\gamma)  
	= \mathcal{E}_{\succ}\big(-W(-\gamma)\big),\quad \gamma \in g_B({A}).
\end{equation}

\begin{thm}\label{thm:shufflefactorization}
For $\gamma_1 ,\gamma_2 \in g_B({A})$ we have 
\begin{equation}
\label{factorization1}
	\mathcal{E}_{\prec}(\gamma_1) * \mathcal{E}_{\prec}( {\gamma_2}^{\gamma_1} ) 
	= \mathcal{E}_{\prec}(\gamma_1 + \gamma_2)
\end{equation}
and
\begin{equation}
\label{factorization2}
	\mathcal{E}_{\succ}({{\gamma_1}^{-\gamma_2}}) *  \mathcal{E}_{\succ}(\gamma_2) 
	= \mathcal{E}_{\succ}(\gamma_1 + \gamma_2)
\end{equation}
where 
\begin{equation}
\label{adaction}
	{\gamma_2}^{\gamma_1} := {ad}^\prec_{\gamma_1}(\gamma_2).
\end{equation}
\end{thm}

\begin{proof}
The proof follows from \eqref{BCHproduc3} together with $\mathcal{E}_{\prec}(\gamma_1) = \exp^*\! \big(\Omega'(\gamma_1)\big)$. Indeed
\begin{eqnarray*}
	\mathcal{E}_{\prec}(\gamma_1) * \mathcal{E}_{\prec}( {\gamma_2}^{\gamma_1}) 
	&=& \exp^*\! \big(\Omega'(\gamma_1)\big) * \exp^*\! \big(\Omega'({\gamma_2}^{\gamma_1})\big)\\
	&=& \exp^*\! \Big(\mathrm{BCH}\big(\Omega'(\gamma_1),\Omega'({\gamma_2}^{\gamma_1})\big)\Big)\\
	&=& \exp^*\! \big(\Omega'(\gamma_1 \# {\gamma_2}^{\gamma_1})\big)\\
	&=& \exp^*\! \big(\Omega'(\gamma_1 + {Ad}_{\mathcal{E}^{-1}_{\prec}(\gamma_1)}({\gamma_2}^{\gamma_1}))\big)\\
	&=& \exp^*\! \big(\Omega'(\gamma_1+\gamma_2)\big) = \mathcal{E}_{\prec}(\gamma_1+\gamma_2). 
\end{eqnarray*}	
To show \eqref{factorization2} we calculate
\begin{eqnarray*}
	\mathcal{E}_{\succ}({\gamma_1}^{-\gamma_2}) * \mathcal{E}_{\succ}(\gamma_2) 
	&=& \exp^*\! \big(-\Omega'(-{\gamma_1}^{-\gamma_2})\big) * \exp^*\! \big(-\Omega'(-\gamma_2)\big)\\
	&=& \exp^*\! \Big(\mathrm{BCH}\big(-\Omega'(-{\gamma_1}^{-\gamma_2}),-\Omega'(-\gamma_2)\big)\big)\Big)\\
	&=& \exp^*\! \Big(-\mathrm{BCH}\big(\Omega'(-\gamma_2),\Omega'(-{\gamma_1}^{-\gamma_2})\big)\big)\Big)\\
	&=& \exp^*\! \big(-\Omega'(-\gamma_2 \# -{\gamma_1}^{-\gamma_2})\big)\\
	&=& \exp^*\! \big(\Omega'(-\gamma_2 - {Ad}_{\mathcal{E}^{-1}_{\prec}(-\gamma_1)}({\gamma_1}^{-\gamma_2}))\big)\\
	&=& \exp^*\! \big(-\Omega'(-\gamma_2-\gamma_1)\big) =\mathcal{E}_{\succ}(\gamma_1+\gamma_2). 
\end{eqnarray*}	
\end{proof}


\section{Universal products in free probability}
\label{sect:universal}

We now consider the three products, i.e., the (non-commutative) shuffle product $*$ and the two commutative products $\boxprec$, $\boxsucc$, from the point of view of non-commutative probability theory. The main result in this section is that (except for the tensor product, related to classical probability) the three universal products in non-commutative probability (related to free, monotone and boolean additive convolution) can be obtained directly from those developed in the previous sections. We refer the reader to \cite{bengohrschur_02,muraki_03,speicher_97b} for background, definitions, and details on the classification and construction of universal products and the relations to additive convolution in free, monotone and boolean probability.

\smallskip

Suppose now that $(\mathcal A_1,\varphi_1)$ and $(\mathcal A_2,\varphi_2)$ are two non-unital non-commutative probability spaces (non-unital meaning that we do not assume that either $\mathcal A_1$ or $\mathcal A_2$ have a unit, and we therefore do not require the linear forms $\varphi_1$ and $\varphi_2$ to be unital). We write, as usual, $\Phi_1$ and $\Phi_2$ for the extensions of $\varphi_1$ respectively $\varphi_2$ to characters on the double tensor algebras $T(T(\mathcal A_1))$ respectively $T(T(\mathcal A_2))$.

The problem of universal products consists in constructing a linear form on the (non-unital) free product $\mathcal A := \mathcal A_1 \star \mathcal A_2$ of the two algebras out of the linear forms $\varphi_1$ and $\varphi_2$. To be universal, the construction has to satisfy functoriality and combinatorial criteria, see \cite{muraki_03,speicher_97b}. Recall that $\mathcal A$ is generated as a vector space by {\it{alternating words}}, that is, sequences of the form $w = x_1 \cdots x_n$, where the $x_i$ with an odd index belong to $\mathcal A_1$ and the $x_j$ with an even index to the algebra $\mathcal A_2$, or conversely, the $x_i$ with an even index belong to $\mathcal A_1$ and the $x_j$ with an odd index belong to the $\mathcal A_2$. We still write $\varphi_1$ respectively $\varphi_2$ for the extensions of $\varphi_1$ and $\varphi_2$ from $\mathcal A_1$ respectively $\mathcal A_2$ to $\mathcal A$, which are null on words that contain letters from $\mathcal \mathcal A_2$ respectively $\mathcal A_1$, and we write $\Phi_1$, $\Phi_2 $ for their extensions to the double tensor algebra $T(T(\mathcal A))$.

\begin{thm}\label{univprod}
The Speicher--Muraki \cite{muraki_03,speicher_97b} boolean, free, monotone and anti-monotone products of $\varphi_1$ and $\varphi_2$ are the restrictions to $\mathcal A \subset T(T(\mathcal A))$ of the characters
\begin{enumerate}[i)]
	\item $\Phi_1\ast \Phi_2$ {\rm{(monotone product)}},

	\item $\Phi_2\ast \Phi_1$ {\rm{(antimonotone product)}}

	\item $\Phi_1\boxprec \Phi_2$ {\rm{(free product)}}

	\item $\Phi_1\boxsucc\Phi_2$ {\rm{(boolean product)}}.
\end{enumerate}
\end{thm}

The proof is divided into a series of lemmas corresponding to the various cases. Each of the formulas obtained in the lemmas corresponds to the definition of one of the universal products, see e.g. \cite{muraki_03}. 

\begin{lem}[Monotone and antimonotone products]\label{prop:monotonProd}
Let $w = x_1 \cdots x_n$ be an alternating word, then
\begin{equation}
\label{leftmonoton}
	\Phi_1 * \Phi_2 (w) =  \Phi_1 (\prod_{x_i\in \mathcal A_1}x_i) \prod_{x_j\in \mathcal A_2} \Phi_2 (x_j)
\end{equation}
\begin{equation}
\label{rightmonoton}
	\Phi_2 * \Phi_1 (w)= \Phi_2 (\prod_{x_i\in \mathcal A_2} x_i) \prod_{x_j\in \mathcal A_1} \Phi_1 (x_j).
\end{equation}
\end{lem}

\begin{proof}
We show \eqref{leftmonoton} when $n=2m$ and the first letter of $w = x_1 \cdots x_n$ is $x_1 \in \mathcal A_1$, the other cases as well as the other identity follow from similar arguments. Recall that 
$$
	\Phi_1 * \Phi_2 = \Phi_1 \succ \Phi_2 + \Phi_1 \prec \Phi_2,
$$
and that the characters $\Phi_1$ and $\Phi_2$ vanish on words that contain letters from  $\mathcal A_2$ respectively $\mathcal A_1$. Notice that, since the alternating word $w$ starts with a letter from $\mathcal A_1$, we have $(\Phi_1 \succ \Phi_2)(w) = 0$. Indeed,  when $(\Phi_1 \succ \Phi_2)(w)$ is expanded using the definition of the right half-unshuffle coproduct, $\Phi_2$ acts on words always containing the letter $x_1$. It follows that the left half-shuffle product $(\Phi_1 \prec \Phi_2)(w)$ is non-zero only in the following case 
$$
	(\Phi_1 \prec \Phi_2)(w)=(\Phi_1 \otimes \Phi_2)(x_1 \cdots x_{2m-1}\otimes x_{2}|x_{4}| \cdots |x_{2m}),
$$ 
which gives \eqref{leftmonoton} since $\Phi_2(x_{2}|x_{4}| \cdots |x_{2m})=\prod_{r = 1}^m \Phi_2 (x_{2r})$. 
\end{proof}

\begin{lem} [Free product]
Let $w=x_1\cdots x_n$ be an alternating word, then, for $\Phi:=\Phi_1 \ \boxprec\Phi_2$ we have
\begin{equation}
\label{FreeUnivProd}
	\Phi (w) = - \sum_{1 \in S \varsubsetneq [2n]} (-1)^{2n - |S|} \Phi(w_S) \prod_{i \notin S \atop x_i\in \mathcal A_1} \Phi_1 (x_i) \prod_{j\notin S \atop x_j\in \mathcal A_2} \Phi_2 (x_j).
\end{equation}
\end{lem}

\begin{proof}
Recall that for $S=\{i_1,\dots,i_k\} \subseteq [2n]$, $w_S$ stands for the word $x_{i_1}\cdots x_{i_k}$. Let us consider the infinitesimal characters $\kappa_1:=\log^\prec(\Phi_1)$ and $\kappa_2:=\log^\prec(\Phi_2)$. Then, $\kappa_1$ and $\kappa_2$ vanish on words that contain letters from $\mathcal A_2$ respectively $\mathcal A_1$. The infinitesimal character  $\kappa := \kappa_1+\kappa_2$,  vanishes therefore on mixed words that have letters from both $\mathcal A_1$ and $\mathcal A_2$. By definition of $\boxprec$ we have that $\Phi = \Phi_1 \ \boxprec\Phi_2$ satisfies
$$
	\Phi = e + \kappa \prec \Phi.
$$
For a (non-empty) alternating word $w$ we find $0=\kappa (w)  = \log^\prec (\Phi)(w) = ( (\Phi-e) \prec \Phi^{-1})(w)$. This yields a recursion for calculating $\Phi(w)$ 
\begin{equation}
\label{FreeConvRecursion}
	 \Phi(w) = - (\Phi \prec \Phi^{-1} \circ P) (w),
\end{equation}
where  $P=\id -e$ is the augmentation projector. We deduce for the alternating word $w$ the expansion
$$
	 \Phi(w) = - \sum_{1 \in S \varsubsetneq [2n]} \Phi(w_S) \Phi^{-1}(w_{J_1}) \cdots \Phi^{-1}(w_{J_k}),
$$
where we used the same notation as in the definition of the coproduct $\Delta$ on $T(T(\mathcal A))$. Notice that by definition of the coproduct, the words $w_{J_i}$ are either single letters, alternating, or they are quasi alternating, i.e., they are alternating with the first and last letter being from the same algebra.

Recall now that $\Phi^{-1} = \mathcal{E}^{-1}_{\prec}(\kappa)=\mathcal{E}_{\succ}(-\kappa)$ solves the right half-shuffle fixed point equation $\Phi^{-1} = e - \Phi^{-1} \succ \kappa$. In particular, on a letter $a$ of $T_1(\mathcal A)$, $\Phi^{-1}(a)= -\kappa(a)$. The vanishing of $\kappa$ on products of words (in $T(T(\mathcal A))$) and on mixed words with letters from both $\mathcal A_1$ and $\mathcal A_2$, implies that for a (quasi) alternating word $v = x_{l_1} \cdots x_{l_p}$ 
\begin{equation}
\label{boolean2}
	\Phi^{-1}(v)=(-1)^p\prod_{x_{l_r}\in \mathcal A_1} \Phi_1 (x_{l_r})\prod_{x_{l_p}\in \mathcal A_2} \Phi_2 (x_{l_p}).
\end{equation}
This yields \eqref{FreeUnivProd}. 
\end{proof}

\begin{lem} [Boolean product]
Let $w=x_1\cdots x_n$ be an alternating word, then
\begin{equation}
\label{BoolUnivProd}
	\Phi_1 \ \boxsucc\Phi_2(w) = \prod_{x_i\in \mathcal A_1} \Phi_1 (x_i)\prod_{x_j \in \mathcal A_2} \Phi_2 (x_j).
\end{equation}
\end{lem}

\begin{proof}
Consider the boolean infinitesimal characters $\beta_1:=\log^\succ(\Phi_1)$ and $\beta_2:=\log^\succ(\Phi_2)$. Again, both $\beta_1$ and $\beta_2$ vanish on words that have letters from both $\mathcal A_1$ and $\mathcal A_2$. The infinitesimal character $\beta := \beta_1+\beta_2$ also vanishes on words that have letters from both $\mathcal A_1$ and $\mathcal A_2$. The corresponding character $\Phi := \Phi_1 \ \boxsucc \Phi_2$ satisfies the right half-shuffle equation
$$
	\Phi = e + \Phi \succ (\beta_1+\beta_2).
$$
The Lemma follows by induction from the usual right half-shuffle properties and $\beta$ being an infinitesimal character.
\end{proof}


\section{The Bercovici--Pata bijection revisited}
\label{sect:BPbijection}

Belinschi and Nica defined a multivariable extension of the Bercovici--Pata bijection \cite{bercovicipata_99}, which is a bijection between the set of joint distributions of non-commutative random variables and the subset of joint distributions which are infinitely divisible. We refer the reader to \cite{belinschinica_08,nica_09} for background, details and further references. In this section we show how the algebraic properties of the bijection can be accounted for and studied using the point of view of half-shuffle logarithms and exponentials. For that purpose, we introduce another bijection that, although living on the group of characters associated to a non-commutative probability space, and not on joint distributions, happens to inherit many properties of the Bercovici--Pata bijection. In fact, it allows to recover the latter.

Let $(\mathcal A,\varphi)$ be a non-commutative probability space. We write as usual $\Phi$ for the corresponding character on $\overline H=\overline T(T(\mathcal A))$,  $\rho$ for the monotone cumulant, $\kappa$ for the free cumulant, and $\beta$ for the boolean cumulant infinitesimal character.

\begin{defn}\label{def:BPbijection}
The group-theoretical Bercovici--Pata bijection is defined to be the set automorphism of $G_H(k)$
$$
	{\mathbb B}(\Phi):={\mathcal E}_\prec\circ \log^\succ (\Phi).
$$
\end{defn}

Let us explain how Definition \ref{def:BPbijection} connects to the classical definition in terms of distributions, and develop also some of its algebraic follow-ups. If $\mathrm{P}$ is a non-commutative monomial (or polynomial) in $k\langle X_1,\dots,X_n\rangle $, we write $\mathrm{P}(a_1,\dots,a_n)$ for its evaluation at $X_1=a_1,\dots,X_n=a_n$.

\begin{defn}
There are four canonical pairings 
$$
	\mathcal A^{\otimes n} \otimes k\langle X_1,\dots,X_n\rangle \to k
$$
associated respectively to the notions of free moment, free cumulant, boolean cumulant, monotone cumulant, and defined respectively by
\begin{itemize}

\item {\rm{(Moment pairing)}} $<a_1\cdots a_n|\mathrm{P}>_\varphi:=\varphi(\mathrm{P}(a_1,\ldots,a_n)),$

\smallskip

\item {\rm{(Monotone pairing)}} $<a_1\cdots a_n|\mathrm{P}>_\rho:=\rho(\mathrm{P}(a_1,\ldots,a_n)),$

\smallskip

\item {\rm{(Free pairing)}} $<a_1\cdots a_n|\mathrm{P}>_\kappa:=\kappa(\mathrm{P}(a_1,\ldots,a_n)),$

\smallskip

\item {\rm{(Boolean pairing)}} $<a_1\cdots a_n|\mathrm{P}>_\beta:=\beta(\mathrm{P}(a_1,\ldots,a_n)).$

\smallskip

\end{itemize}\end{defn}

Pairings with the monomial $\mathrm{P}=X_1\cdots X_n$ define the usual mixed moment, mixed free cumulant, mixed boolean cumulant, and mixed monotone cumulant associated to the free random variables $a_1,\dots ,a_n$.

\begin{defn}
Conversely, each sequence ${\mathbf a}=(a_1,\dots,a_k)$ defines four linear forms (functionals) on $k\langle X_1,\dots,X_n \rangle$:
\begin{itemize}
\item {\rm{(Moment functional)}} $\varphi_{\mathbf a}(\mathrm{P}):=\varphi(\mathrm{P}(a_1,\dots,a_n)),$

\smallskip

\item {\rm{(Monotone cumulant functional)}} $\rho_{\mathbf a}(\mathrm{P}):=\rho(\mathrm{P}(a_1,\dots,a_n)),$

\smallskip

\item {\rm{(Free cumulant functional)}} $\kappa_{\mathbf a}(\mathrm{P}):=\kappa(\mathrm{P}(a_1,\dots,a_n)),$

\smallskip

\item {\rm{(Boolean cumulant functional)}} $\beta_{\mathbf a}(\mathrm{P}):=\beta(\mathrm{P}(a_1,\dots,a_n)).$
\end{itemize}
\end{defn}

As a useful lemma we state a consequence of the identities relating the character $\Phi$ and the various cumulant maps.

\begin{lem}
The exponential relations between the character $\Phi$ and the infinitesimal characters $\rho$, $\kappa$, and $\beta$ imply that any of these four functionals determines completely the others. In particular, the three cumulant functionals depend only on the (unital) functional $\mu:=\phi_{\mathbf a}$ (also referred to as the distribution of $\mathbf a$), and not on the actual sequence $\mathbf a$ of non-commutative random variables.
\end{lem}

Although equivalent to dealing directly with the functionals, it is sometimes convenient to introduce their generating series.
Following \cite{nica_09}, we define the \it series of moments \rm of ${\mathbf a}=(a_1,\dots,a_k)$ by
$$
	M_{\mathbf a}=M_{\mu}:=\sum_w \varphi_{\mathbf a}(w) w,
$$
where the sum runs over all words (i.e., non-commutative monomials) over the alphabet $\{X_1,\dots,X_n\}$. The $R$-transform is defined similarly by 
$$
	R_{\mathbf a}=R_\mu:=\sum_w \kappa_{\mathbf a}(w) w,
$$
and the so-called $\eta$-series by
$$
	\eta_{\mathbf a}=\eta_\mu:=\sum_w \beta_{\mathbf a}(w) w.
$$
Notice that these definitions are purely algebraic and make sense for an arbitrary functional $\mu$ (i.e., not necessarily associated to a sequence $\mathbf a$ of elements of $\mathcal A$). For later use, to emphasise the dependency of these definitions on the choice of a linear form $\varphi$ on $\mathcal A$, we will also write $M^\varphi_{\mu}$ or $M^\Phi_{\mu}$ for $M_\mu$, and so on.

The relations between $\Phi$ and the various cumulants translate into relations between their generating series. For example (compare with \cite{belinschinica_08}), the identity $\Phi = e + \Phi \succ \beta$ translates (using the definition of the right half-shuffle $\succ$ and the fact that $\beta$ is an infinitesimal character) into
$$
	M_\mu=\eta_\mu+M_\mu\cdot \eta_\mu.
$$

The Bercovici--Pata bijection $\mathbb{B}$ is then defined in this context by the relation \cite{belinschinica_08}
\begin{equation}\label{bprel}
	R_{\mathbb{B}(\mu)}=\eta_\mu.
\end{equation}

From the general properties of half-shuffle logarithms, we get the following relation between  the group-theoretically defined ${\mathbb B}(\Phi)$ and the Bercovici--Pata bijections.

\begin{prop}\label{bpfundamental}
For any sequence $\mathbf a$ of non-commutative random variables in $(\mathcal A,\varphi)$ with moment functional $\mu$, we have
$$
	\log^\prec(\mathbb{B}(\Phi))({\mathbf a})=\log^\succ (\Phi) ({\mathbf a}),
$$
and therefore
$$
	R_{\mathbb{B}(\mu)}=R^{\mathbb{B}(\Phi)}_{\mathbf a}=\eta_{\mathbf a}=\eta_\mu.
$$
\end{prop}

Motivated by \cite{belinschinica_09}, let us also introduce

\begin{defn}
The one-parameter Bercovici--Pata semigroup of set automorphisms of $G_H(k)$ is defined for $t\geq 0$
$$
	\mathbb{B}_t(\Phi) = 
	\big(\mathcal{E}_{\prec}((1+t)\log^\succ\Phi)\big)^{{\bboxsucc}\frac{1}{(1+t)}}.
$$
\end{defn}

Let us show that these set automorphisms form indeed a $1$-parameter semigroup, i.e., $\mathbb{B}_t \circ \mathbb{B}_s(\gamma)=\mathbb{B}_{t+s}(\gamma),$ where $\gamma:=\log^\succ\Phi$. Recall that $\gamma_2^{\gamma_1}:={ad}^\prec_{\gamma_1}(\gamma_2)$. First, we have

\begin{lem}\label{lem:simple1} 
For $t>0$ we have that 
$$
	\mathbb{B}_t(\Phi) = \big(\mathcal{E}_{\prec}(t\gamma)\big)^{{\bboxsucc}\frac{1}{t}}.
$$ 
\end{lem}

\begin{proof}
This follows from 
$$
	\mathcal{E}_{\prec}\big((1+t)\gamma\big) 
	= \mathcal{E}_{\prec}(\gamma) \boxprec \mathcal{E}_{\prec}(t\gamma)
	= \mathcal{E}_{\prec}(\gamma) * \mathcal{E}_{\prec}(t\gamma^\gamma),
$$
such that, using Corollary \ref{cor:eta-series}, $\mathcal{E}_{\prec}\big((1+t)\gamma\big) = \mathcal{E}_{\succ}\big( \log^\succ \big( \mathcal{E}_{\prec}\big((1+t)\gamma\big)\big)\big)$ calculates
\begin{align*}
	\mathcal{E}_{\prec}\big((1+t)\gamma\big) 
	&= \mathcal{E}_{\succ}\Big(\mathcal{E}^{-1}_{\prec}\big((1+t)\gamma\big) \succ (1+t) \gamma
	\prec \mathcal{E}_{\prec}\big((1+t)\gamma\big)\Big)\\
	&=  \mathcal{E}_{\succ}\Big((t+1) \mathcal{E}^{-1}_{\prec}(t\gamma^\gamma) \succ  \gamma^\gamma
	\prec \mathcal{E}_{\prec}(t\gamma^\gamma)\Big).
\end{align*}
This then implies for $t>0$ that
\begin{align*}
	\mathbb{B}_t(\Phi) 
	&= \big(\mathcal{E}_{\prec}\big((1+t)\gamma\big)\big)^{{\bboxsucc}\frac{1}{1+t}}\\
	&= \Big(\mathcal{E}_{\succ}\big((t+1) \mathcal{E}^{-1}_{\prec}(t\gamma^\gamma) \succ  \gamma^\gamma
	\prec \mathcal{E}_{\prec}(t\gamma^\gamma)\big)\Big)^{{\bboxsucc}\frac{1}{1+t}}\\
	&= \mathcal{E}_{\succ}\Big(\mathcal{E}^{-1}_{\prec}(t\gamma^\gamma) \succ  \gamma^\gamma
	\prec \mathcal{E}_{\prec}(t\gamma^\gamma)\Big)\\
	&= \mathcal{E}_{\succ}\Big(\frac{1}{t}\mathcal{E}^{-1}_{\prec}(t\gamma^\gamma) \succ  t\gamma^\gamma
	\prec \mathcal{E}_{\prec}(t\gamma^\gamma)\Big)\\
	&= \mathcal{E}_{\succ}\Big(\frac{1}{t}\log^\succ\big(\mathcal{E}_{\prec}(t\gamma^\gamma)\big)\Big)\\
	&= \mathcal{E}_{\succ}\Big(\log^\succ\big(\mathcal{E}_{\prec}(t\gamma^\gamma)^{\bboxsucc \frac{1}{t}}\big)\Big)\\
	&=\mathcal{E}_{\prec}(t\gamma^\gamma)^{\bboxsucc \frac{1}{t}}.
\end{align*}
\end{proof}

Hence, we get for $t>0$
$$
	\mathbb{B}_t(\Phi)  = \mathcal{E}_{\prec}(t\gamma^\gamma)^{\bboxsucc \frac{1}{t}}
				   =  \mathcal{E}_{\prec}\big(\gamma^{t\gamma}\big) .
$$ 
The second equality follows from a similar calculation using $\mathcal{E}_{\prec}\big((1+t)\gamma\big) = \mathcal{E}_{\prec}(t\gamma) * \mathcal{E}_{\prec}(\gamma^{t\gamma})$, which follows from commutativity. From this it is evident that $\mathbb{B}_1(\Phi) = \mathbb{B}(\Phi)$. 

Using a similar line of calculation we now show the $1$-parameter semigroup property of $\mathbb{B}_t$. Indeed, we have for $t,s>0$  
\begin{align*}
	\mathbb{B}_t \circ \mathbb{B}_s (\Phi) 
	&= \mathbb{B}_t\big(\gamma^{s\gamma}\big) \\
	&=  \mathcal{E}_{\prec}\big((\gamma^{s\gamma})^{t\gamma^{s\gamma}}\big). 
\end{align*}
The righthand side of the last equality requires some calculation. We obtain for $(\gamma^{s\gamma})^{t\gamma^{s\gamma}} = {ad}^\prec_{{t\gamma^{s\gamma}}}(\gamma^{s\gamma})$,
\begin{align*}
	(\gamma^{s\gamma})^{t\gamma^{s\gamma}}
	&= \mathcal{E}^{-1}_{\prec}\big({t\gamma^{s\gamma}}\big)
	\succ \gamma^{s\gamma} \prec
	\mathcal{E}_{\prec}\big({t\gamma^{s\gamma}}\big)\\
	&= \mathcal{E}^{-1}_{\prec}\big({t\gamma^{s\gamma}}\big)
	\succ \Big(\mathcal{E}^{-1}_{\prec}\big(s \gamma\big) 
	\succ  \gamma \prec 
	\mathcal{E}_{\prec}\big(s \gamma\big)\Big)\prec
	\mathcal{E}_{\prec}\big({t\gamma^{s\gamma}}\big)\\ 
	&= \big(\mathcal{E}^{-1}_{\prec}\big({t\gamma^{s\gamma}}\big) *
	\mathcal{E}^{-1}_{\prec}\big(s \gamma\big) 
	 \big)\succ \gamma \prec \big(
	\mathcal{E}_{\prec}\big(s \gamma\big)*
	\mathcal{E}_{\prec}\big({t\gamma^{s\gamma}}\big)\big)\\ 
	&=\mathcal{E}^{-1}_{\prec}\big({(t+s)\gamma}\big)  
	 \succ \gamma \prec 
	\mathcal{E}_{\prec}\big((t+s) \gamma\big)\\
	&= \gamma^{(t+s)\gamma}.
\end{align*}
Hence, this yields 
$$
	\mathbb{B}_t \circ \mathbb{B}_s (\Phi)
	= \mathcal{E}_{\prec}\big(\gamma^{(t+s)\gamma}\big)
	= \mathbb{B}_{t +s} (\Phi).
$$
A more detailed study of the group-theoretical one-parameter Bercovici--Pata semigroup from the shuffle algebra viewpoint will be presented elsewhere.


\section{Free additive convolution and subordination products}
\label{sect:additiveconvol}

This section is motivated by the theory of additive convolution of distributions of non-commutative random variables. It provides a different, more algebraic and structural, point of view on the problems addressed in Nica's paper \cite{nica_09}, which, in turn, is related to the references \cite{belinschinica_08,belinschinica_09,bercovicipata_99,lenczewski_07}. Here, we will focus on free additive convolution, as in \cite{nica_09}, the other convolution products could be addressed similarly. Concretely, we show that the theory can be lifted from the study of distributions of families of non-commutative random variables to the one of operations on the group of characters $G_H(k)$. 

Recall first that, given two non-commutative probability spaces $(\mathcal A_1,\varphi_1)$, $(\mathcal A_2,\varphi_2)$, we introduced four ways to extend the corresponding characters $\Phi_1$ and $\Phi_2$ to a character $\Phi$ (given respectively by $\Phi_1\ast\Phi_2$, $\Phi_2\ast \Phi_1$, $\Phi_1\boxprec \Phi_2$, $\Phi_1\boxsucc\Phi_2$) on the free product $\mathcal A = \mathcal A_1\star \mathcal A_2$. Free additive convolution corresponds to the third product, $\Phi_1\boxprec \Phi_2$, in the following sense. 

Let $\mathbf a$ and $\mathbf b$ be two sequences of $n$ elements of $\mathcal A_1$ and $\mathcal A_2$, with moment functionals $\mu_1$ and $\mu_2$. Then, by definition, the free additive convolution $\mu_1\boxplus\mu_2$ of $\mu_1$ and $\mu_2$ is characterised by the fact that it linearises the $R$-transform in the sense that
$$
	R_{\mu_1\boxplus\mu_2}=R_{\mu_1}+R_{\mu_2}.
$$
Equivalently, in our approach, it is the moment functional of $\mathbf a+\mathbf b$ for the character $\Phi=\Phi_1\boxprec \Phi_2$. Recall that the latter is given by $\mathcal{E}_{\prec}(\gamma_1+\gamma_2)$ with $\gamma_i=\log^\prec(\Phi_i), \ i=1,2$, so that 
$$
	R^{\Phi}_{\mathbf a+\mathbf b}=(\gamma_1+\gamma_2)(\mathbf a+\mathbf b)
	=\gamma_1(\mathbf a)+\gamma_2(\mathbf b)
	=R^{\Phi_1}_{\mathbf a}+R^{\Phi_2}_{\mathbf b}.
$$

From now on we let $\Phi_1$ and $\Phi_2$ be two arbitrary characters in $G_B(k)$ and $\gamma_i=\log^\prec(\Phi_i),\ i=1,2$.

\begin{defn}\label{def:subordination}
The {\rm{left subordination product}} of $\Phi_1=\mathcal{E}_{\prec}(\gamma_1)$ and $\Phi_2=\mathcal{E}_{\prec}(\gamma_2)$   is defined by
\begin{equation}
\label{sproduct1}
	\Phi_2 \boxright \Phi_1
	:= \mathcal{E}_{\prec}(ad_{\gamma_1}^\prec(\gamma_2))=
	\mathcal{E}_{\prec}( {\gamma_2}^{\gamma_1}).
\end{equation}
The {\rm{right subordination product}} of $\Psi_1=\mathcal{E}_{\succ}(\gamma_1)$ and $\Psi_2=\mathcal{E}_{\succ}(\gamma_2)$   is defined by
\begin{equation}
\label{sproduct2}
	\Psi_1 \boxleft \Psi_2:= \mathcal{E}_{\succ}(ad_{-\gamma_1}^\prec(\gamma_2))	
	= \mathcal{E}_{\succ}({\gamma_2}^{-\gamma_1}).
\end{equation}
\end{defn}

For example, from this definition we obtain the following expression for the set automorphism $\mathcal{E}_{\prec}\circ\log^\succ$, i.e., the group-theoretical Bercovici--Pata bijection:
\begin{lem}
\begin{equation}
\label{BBPBshuffle}
	\mathbb{B}(\Phi)=\mathcal{E}_{\prec}\circ\log^{\succ}\big(\Phi\big)
				 = \mathcal{E}_{\prec}( {\gamma}^{\gamma})
				 = \Phi \boxright \Phi	.
\end{equation}
\end{lem}
For the additive convolution products of characters, we derive the next result.
\begin{lem}
We have
$$
	\Phi_1\boxprec \Phi_2
	=  \Phi_1 * \big(\Phi_2 \boxright \Phi_1 \big) 
	=  \Phi_2 * \big(\Phi_1   \boxright \Phi_2 \big), 
$$
$$
	\Phi_1\boxsucc \Phi_2 
	= (\Phi_2 \boxleft  \Phi_1)  * \Phi_2
	= (\Phi_1 \boxleft  \Phi_2) * \Phi_1.
$$
\end{lem}

Indeed, from the commutativity of $\boxprec$ and $\boxsucc$ and Theorem \ref{thm:shufflefactorization} we obtain
\begin{equation}
\label{decomp}
	\mathcal{E}_{\prec}(\gamma_1)  \ \boxprec \mathcal{E}_{\prec}(\gamma_2)  
	= \mathcal{E}_{\prec}(\gamma_1) 
	* \big(\mathcal{E}_{\prec}(\gamma_2)   \boxright\mathcal{E}_{\prec}(\gamma_1) \big) 
	=  \mathcal{E}_{\prec}(\gamma_2) 
	* \big(\mathcal{E}_{\prec}(\gamma_1)   \boxright \mathcal{E}_{\prec}(\gamma_2) \big), 
\end{equation}
and analogously for 
$$
	\mathcal{E}_{\succ}(\gamma_1)\ \boxsucc \mathcal{E}_{\succ}(\gamma_2) 
	= \big(\mathcal{E}_{\succ}(\gamma_2) \boxleft \mathcal{E}_{\succ}(\gamma_1)\big) 
	* \mathcal{E}_{\succ}(\gamma_2)
	= \big(\mathcal{E}_{\succ}(\gamma_1) \boxleft  \mathcal{E}_{\succ}(\gamma_2)\big) 
	* \mathcal{E}_{\succ}(\gamma_1).
$$ 

Various other identities follow immediately from the fundamental group- and Lie-theoretical identities that hold on $G_B(k)$ respectively $g_B(k)$. Notice that these identities, although similar to the ones that have been obtained for the free additive convolution of non-commutative random variables, \it do not assume \rm that we are dealing with freely independent distributions. They hold for arbitrary characters $\Phi_1$, $\Phi_2$ and infinitesimal characters. In that sense they are a true generalisation of the formulas that can be obtained when dealing with distributions. We list below some of the most important formulas that can be derived that way.

\begin{lem}[Distributivity of subordination products]
\begin{equation}
\label{distrib}
	\big(\mathcal{E}_{\prec}(\gamma_1)  \ \boxprec \mathcal{E}_{\prec}(\gamma_2) \big)
	\boxright \mathcal{E}_{\prec}(\gamma_3)  
	= \big(\mathcal{E}_{\prec}(\gamma_1)  \boxright \mathcal{E}_{\prec}(\gamma_3) \big) 
	\ \boxprec \big(\mathcal{E}_{\prec}(\gamma_2) \boxright \mathcal{E}_{\prec}(\gamma_3) \big). 
\end{equation}
\end{lem}

This follows immediately from the linearity of \eqref{adaction}, i.e., $({\gamma_1+\gamma_1})^{\gamma_3} = {\gamma_1}^{\gamma_3} + {\gamma_2}^{\gamma_3}$. An analogous statement holds for $\mathcal{E}_{\succ}(\gamma_3) \boxleft \big(\mathcal{E}_{\succ}(\gamma_1)\ \boxsucc \mathcal{E}_{\succ}(\gamma_2)\big)$.

Several computational observations are now listed regarding the product \eqref{sproduct1}. Similar results hold for the product \eqref{sproduct2}. First we show that
\begin{align}
	\mathcal{E}_{\prec}({\gamma_1 + \gamma_2}) 
	&=\mathcal{E}_{\prec}(\gamma_1)  \ \boxprec \mathcal{E}_{\prec}(\gamma_2)  \\
	&= \big(\mathcal{E}_{\prec}(\gamma_1)  \boxright \mathcal{E}_{\prec}(\gamma_2) \big) 
	\ \boxsucc \big(\mathcal{E}_{\prec}(\gamma_2) \boxright \mathcal{E}_{\prec}(\gamma_1) \big). \label{changeprod}
\end{align}
This follows from 
$$
	\mathcal{E}_{\prec}(\gamma_2) \boxright \mathcal{E}_{\prec}(\gamma_1) 
	= \mathcal{E}_{\prec}( {\gamma_2}^{\gamma_1}) 
	= \mathcal{E}_{\succ}( ({\gamma_2}^{\gamma_1})^{{\gamma_2}^{\gamma_1}}), 
$$
which yields  
\begin{align*}
	\big(\mathcal{E}_{\prec}(\gamma_1) \boxright \mathcal{E}_{\prec}(\gamma_2) \big) 
	\ \boxsucc \big(\mathcal{E}_{\prec}(\gamma_2) \boxright \mathcal{E}_{\prec}(\gamma_1)\big)
	&=
	\mathcal{E}_{\succ}\big( ({\gamma_1}^{\gamma_2})^{{\gamma_1}^{\gamma_2}}\big)	 
	\ \boxsucc
	 \mathcal{E}_{\succ}\big( ({\gamma_2}^{\gamma_1})^{{\gamma_2}^{\gamma_1}}\big)\\
	 &=
	 \mathcal{E}_{\succ}\big(({\gamma_1}^{\gamma_2})^{{\gamma_1}^{\gamma_2}}
	 +({\gamma_2}^{\gamma_1})^{{\gamma_2}^{\gamma_1}}\big).
\end{align*}
Next we show that $({\gamma_1}^{\gamma_2})^{{\gamma_1}^{\gamma_2}}= {\gamma_1}^{{\gamma_1}+{\gamma_2}}$
\begin{align*}
	({\gamma_1}^{\gamma_2})^{{\gamma_1}^{\gamma_2}}
	&= \mathcal{E}^{-1}_{\prec}({\gamma_1}^{\gamma_2}) \succ
	{\gamma_1}^{\gamma_2} \prec \mathcal{E}_{\prec}({\gamma_1}^{\gamma_2})\\
	&=\big(\mathcal{E}^{-1}_{\prec}({\gamma_1}^{\gamma_2})
	* \mathcal{E}^{-1}_{\prec}({\gamma_2})\big)
	\succ
	{\gamma_1} 
	\prec  
	\big(\mathcal{E}_{\prec}({\gamma_2}) 
	* \mathcal{E}_{\prec}({\gamma_1}^{\gamma_2})\big)\\
	&=  \mathcal{E}^{-1}_{\prec}({\gamma_1}+{\gamma_2}) 
	\succ
	{\gamma_1}\prec
	  \mathcal{E}_{\prec}({\gamma_1}+{\gamma_2}) \\
	&= {\gamma_1}^{{\gamma_1}+{\gamma_2}}.
\end{align*}
Here we used shuffle axioms \eqref{A1}-\eqref{A3}. Analogously $({\gamma_1}^{\gamma_2})^{{\gamma_1}^{\gamma_2}}
	= {\gamma_1}^{{\gamma_1}+{\gamma_2}}.$ This yields
$$
	\mathcal{E}_{\succ}\big(({\gamma_1}^{\gamma_2})^{{\gamma_1}^{\gamma_2}}
	 +({\gamma_2}^{\gamma_1})^{{\gamma_2}^{\gamma_1}}\big)
	 = \mathcal{E}_{\succ}\big(({\gamma_1 + \gamma_2})^{{\gamma_1}+{\gamma_2}}\big), 
$$
where we used again linearity of \eqref{adaction}. It follows by the same arguments that 
$$
	\mathcal{E}_{\succ}\big(({\gamma_1 + \gamma_2})^{{\gamma_1}+{\gamma_2}}\big)	 
	= \mathcal{E}_{\prec}\Big(\big(({\gamma_1 + \gamma_2})^{{\gamma_1}
	+{\gamma_2}})^{-({\gamma_1 + \gamma_2})^{{\gamma_1}+{\gamma_2}}}\Big).
$$

Then, one can show that 
\begin{align*}
	\lefteqn{\big(({\gamma_1 + \gamma_2})^{{\gamma_1}+{\gamma_2}}\big)^{-({\gamma_1 
	+ \gamma_2})^{{\gamma_1}+{\gamma_2}}}
	=\mathcal{E}_{\prec} \big(({\gamma_1 
	+ \gamma_2})^{{\gamma_1}+{\gamma_2}}\big)  
	\succ
	({\gamma_1 + \gamma_2})^{{\gamma_1}+{\gamma_2}}
	\prec
	\mathcal{E}^{-1}_{\prec} \big(({\gamma_1 
	+ \gamma_2})^{{\gamma_1}+{\gamma_2}}\big)} \\
	&= \big(\mathcal{E}_{\prec} \big(({\gamma_1 
	+ \gamma_2})^{{\gamma_1}+{\gamma_2}}\big)  
	* 
	\mathcal{E}^{-1}_{\prec} ({{\gamma_1}+{\gamma_2}})\big)
	\succ
	({\gamma_1 + \gamma_2})
	\prec
	\big(\mathcal{E}_{\prec} ({{\gamma_1}+{\gamma_2}})*
	\mathcal{E}^{-1}_{\prec} (({\gamma_1 
	+ \gamma_2})^{{\gamma_1}+{\gamma_2}})\big)\\
	&={\gamma_1 + \gamma_2},
\end{align*}
which implies the result
$$
	\mathcal{E}_{\succ}\big(({\gamma_1 + \gamma_2})^{{\gamma_1}+{\gamma_2}}\big)	 
	=\mathcal{E}_{\prec}({\gamma_1 + \gamma_2}).
$$

Finally, we show that 
\begin{equation}
\label{logadditive}
	\log^\succ\big(\mathcal{E}_{\prec}(\gamma_1+\gamma_2)\big) 
	= \log^\succ\big(\mathcal{E}_{\prec}(\gamma_1^{\gamma_2})\big) 
		+ \log^\succ\big(\mathcal{E}_{\prec}(\gamma_2^{\gamma_1})\big) .
\end{equation}
This follows from
\begin{align*}
	\mathcal{E}_{\prec}(\gamma_1+\gamma_2)
	&= \mathcal{E}_{\succ}\big(({\gamma_1 + \gamma_2})^{{\gamma_1}+{\gamma_2}}\big)\\
	&= \mathcal{E}_{\succ}\big((\gamma_1^{\gamma_2})^{\gamma_1^{\gamma_2}}
	+ (\gamma_2^{\gamma_1})^{\gamma_2^{\gamma_1}}\big)\\
	&= \mathcal{E}_{\succ}\big(
	\log^\succ\big(\mathcal{E}_{\prec}(\gamma_1^{\gamma_2})\big) 
	+\log^\succ\big(\mathcal{E}_{\prec}(\gamma_2^{\gamma_1})\big) \big).
\end{align*}
Applying $\log^\succ$ on both sides gives \eqref{logadditive}.

These identities can be applied to joint distributions of freely independent families of non-commutative random variables. For example, identity \eqref{logadditive} implies (see \cite{nica_09})
\begin{equation}
\label{logadditive2}
	\eta_{\mu\boxplus\nu}
	= \eta_{\mu \boxright\nu}
	+ \eta_{\nu \boxright \mu}.
\end{equation}


\end{document}